\documentclass[11pt, a4paper]{amsart}
\usepackage{fullpage}
\usepackage{amsmath}
\usepackage{amssymb}
\usepackage{amsfonts}
\usepackage{mathrsfs}
\usepackage{amscd}
\usepackage{graphicx}
\usepackage[shortlabels]{enumitem}
\usepackage{mathtools}
\usepackage{tikz-cd}
\usepackage{hyperref}
\hypersetup{
    colorlinks=true,
    linkcolor = blue,
    citecolor= blue,
    urlcolor=cyan,
}
\usepackage{pstricks-add}
\usepackage{pgf,tikz}
\usepackage{subcaption}
\usepackage{bm}

\usepackage{tkz-tab}
\usepackage{xpatch}
\xpatchcmd{\tkzTabLine}{$0$}{$\bullet$}{}{}
\tikzset{t style/.style={style=solid}}

\usetikzlibrary{arrows}
\usetikzlibrary{calc}
\usepackage{amsthm}
\usepackage{aliascnt}
\usepackage[capitalise]{cleveref}

\newtheorem{theorem}{Theorem}[section]
\crefname{theorem}{theorem}{theorems}
\Crefname{theorem}{Theorem}{Theorems}

\makeatletter
\newtheoremstyle{onlyname}
  {6pt}
  {6pt}
  {\itshape}
  {}
  {\bfseries}
  {.}
  {0.5em}
  {\thmnote{#3}}
\makeatother

\theoremstyle{onlyname}
\newaliascnt{namedtheoreminner}{theorem}
\newtheorem{namedtheoreminner}[namedtheoreminner]{}
\aliascntresetthe{namedtheoreminner}
\crefname{namedtheoreminner}{theorem}{theorems}
\Crefname{namedtheoreminner}{Theorem}{Theorems}

\theoremstyle{plain}

\newaliascnt{proposition}{theorem}
\newtheorem{proposition}[proposition]{Proposition}
\aliascntresetthe{proposition}
\crefname{proposition}{proposition}{propositions}
\Crefname{proposition}{Proposition}{Propositions}

\newaliascnt{lemma}{theorem}
\newtheorem{lemma}[lemma]{Lemma}
\aliascntresetthe{lemma}
\crefname{lemma}{lemma}{lemmas}
\Crefname{lemma}{Lemma}{Lemmas}

\newaliascnt{corollary}{theorem}
\newtheorem{corollary}[corollary]{Corollary}
\aliascntresetthe{corollary}
\crefname{corollary}{corollary}{corollaries}
\Crefname{corollary}{Corollary}{Corollaries}

\newaliascnt{conjecture}{theorem}

\aliascntresetthe{conjecture}
\crefname{conjecture}{conjecture}{conjectures}
\Crefname{conjecture}{Conjecture}{Conjectures}

\newaliascnt{claim}{theorem}

\aliascntresetthe{claim}
\crefname{claim}{claim}{claims}
\Crefname{claim}{Claim}{Claims}

\newaliascnt{question}{theorem}

\aliascntresetthe{question}
\crefname{question}{question}{questions}
\Crefname{question}{Question}{Questions}

\theoremstyle{definition}

\newaliascnt{procedure}{theorem}

\aliascntresetthe{procedure}
\crefname{procedure}{procedure}{procedures}
\Crefname{procedure}{Procedure}{Procedures}

\theoremstyle{definition}

\newaliascnt{algorithm}{theorem}

\aliascntresetthe{algorithm}
\crefname{algorithm}{algorithm}{algorithms}
\Crefname{algorithm}{Algorithm}{Algorithms}

\newaliascnt{example}{theorem}
\newtheorem{example}[example]{Example}
\aliascntresetthe{example}
\crefname{example}{example}{examples}
\Crefname{example}{Example}{Examples}

\newaliascnt{definition}{theorem}

\aliascntresetthe{definition}
\crefname{definition}{definition}{definitions}
\Crefname{definition}{Definition}{Definitions}

\newaliascnt{remark}{theorem}
\newtheorem{remark}[remark]{Remark}
\aliascntresetthe{remark}
\crefname{remark}{remark}{remarks}
\Crefname{remark}{Remark}{Remarks}

\newaliascnt{fact}{theorem}

\aliascntresetthe{fact}
\crefname{fact}{fact}{facts}
\Crefname{fact}{Fact}{Facts}

\crefformat{section}{\S#2#1#3} 
\crefformat{subsection}{\S#2#1#3}
\crefformat{subsubsection}{\S#2#1#3}
\crefname{figure}{{\sc Figure}}{{\sc Figure}}
\crefname{equation}{}{}

\usetikzlibrary{shapes, arrows, calc, arrows.meta, fit, positioning, quotes} 
\tikzset{  
    state/.style ={ellipse, draw, minimum width = 0.9 cm}, 
    point/.style = {circle, draw, inner sep=0.18cm, fill, node contents={}},  
    bidirected/.style={Latex-Latex,dashed}, 
    el/.style = {inner sep=2.5pt, align=right, sloped}  
}  

\colorlet{ColorGray}{gray!10}

\newcommand{\N}{\mathbb{N}}

\newcommand{\Q}{\mathbb{Q}}
\newcommand{\R}{\mathbb{R}}

\newcommand{\cD}{\mathcal{D}}

\newcommand{\cJ}{\mathcal{J}}

\newcommand{\cO}{\mathcal{O}}

\newcommand{\cQ}{\mathcal{Q}}
\newcommand{\cR}{\mathcal{R}}
\newcommand{\cS}{\mathcal{S}}

\newcommand{\tQ}{\widetilde{Q}}

\newcommand{\tS}{\widetilde{S}}

\newcommand{\teta}{\widetilde{\teta}}

\newcommand{\e}{\varepsilon}

\renewcommand{\=}{ \coloneqq }

\DeclareMathOperator{\Dist}{Dist} 

\DeclareMathOperator{\interior}{int}

\DeclareMathOperator{\Cr}{Cr}

\newcommand{\bfw}{\mathbf{w}}
\newcommand{\bfu}{\mathbf{u}}

\newcommand*{\mathinhead}[2]{\texorpdfstring{\boldmath$#1$}{#2}}
\tikzset{
  declare function={
    sgn(\x) = (and(\x<0, 1) * -1) +
    (and(\x>0, 1) * 1) +
    (and(\x==0, 1) * 0);
  }
}

\title[Long-branched Sturmian unimodal maps]
{Absolutely continuous invariant measures and the Collet--Eckmann condition for long-branched Sturmian unimodal maps}

\author{Jorge Olivares-Vinales} \thanks{} 
\address{Shanghai Center for Mathematical Sciences, Jiangwan Campus, Fudan University, No 2005 Songhu Road, Shanghai, China 200438 }
\email{jolivaresv@fudan.edu.cn}


\begin{document}

\begin{abstract}
We study long-branched unimodal maps for which the sequence of values of the kneading map form a Sturmian word. We identify these kneading classes with those arising from the stunted Lorenz construction of Anu\v{s}i\'c, Bruin, and \v{C}in\v{c}, and their quadratic representatives with the maps constructed by Bl\'e from rigid rotations. We compute the cutting and co-cutting times explicitly and establish criteria for the existence and nonexistence of absolutely continuous invariant probability measures. We also show that, no $S$-unimodal realization, with finite critical order, satisfies the Collet--Eckmann condition.
\end{abstract}

\maketitle


\section{Introduction}
\label{sec:introduction}

Unimodal maps form one of the simplest classes of noninvertible dynamical systems in which expansion, recurrence, and criticality interact. Despite their elementary definition, they exhibit a wide range of nonlinear phenomena, and the quadratic family has played a central role in the development of one-dimensional dynamics. The kneading theory of Milnor and Thurston \cite{Milnor_Thurston1986_On_iterated_maps_of_the_interval} and the Markov extension introduced by Hofbauer \cite{Hofbauer-1980_The_top_entrpy_of_axx-1} provide fundamental tools for organizing this dynamics through the itinerary of the turning point and for relating its combinatorics to metric and statistical properties.

An important class in this theory is formed by the \emph{long-branched} unimodal maps introduced and studied by Bruin \cite{Bruin1994_top_condition_for_the_existence_of_acip,Bruin1995_induced_maps_markov_partitions}. These are characterized by bounded kneading map, or equivalently by a uniform positive lower bound on the lengths of the domains of the Hofbauer tower. Long-branched maps lie at the opposite combinatorial extreme from the extensively studied regime $Q(k)\to+\infty$, which includes Fibonacci and Fibonacci-like maps. Bruin constructed a natural Markov inducing scheme for the long-branched class and showed that finiteness of the lifted invariant measure is governed by the integrability of the inducing time. Thus bounded kneading maps retain a strong Markov structure while allowing nontrivial recurrence and statistical behavior.

Within the long-branched regime, irrational rotations give rise naturally to a family with Sturmian combinatorics. Anu\v{s}i\'c, Bruin, and \v{C}in\v{c} introduced these classes through stunted Lorenz circle maps, where the dynamics is organized by an irrational outside rotation number \cite{Anusic-Bruin-Cinic2020_top_properties_of_lorenz_maps}. Alvin and \v{C}in\v{c} subsequently developed their symbolic and recurrence theory in greater detail \cite{alvin-Cinic2026-Recurrence_symbolic_dynamics_and_wild_attractors}. In particular, they showed that long-branched dynamics is incompatible with persistent recurrence of a recurrent nonperiodic turning point, while several other forms of recurrence, including uniform and regular recurrence, can occur.

Independently, Bl\'e encountered the same long-branched rotation structure in the real quadratic family \cite{Ble_2204_external_arguments_and_invariant_measures}. Motivated by the relation between rotation parameters on the main cardioid and the real trace of the Mandelbrot set, he constructed the rational-angle maps explicitly and obtained the irrational-angle representatives as limits along continued-fraction convergents. His approach connects the real dynamics of these maps with rigid rotations and with the geometry of quadratic parameter space.

The present paper gives an intrinsic kneading-theoretic description of these rotation-induced long-branched classes. We show that they are precisely the unimodal kneading classes for which $Q(k)\leq1$ and the binary sequence of values $Q(1)Q(2)Q(3)\cdots$ is Sturmian; equivalently, the kneading map is given by a mechanical coding of an irrational rotation. This identifies the constructions of Anu\v{s}i\'c, Bruin, and \v{C}in\v{c} and of Bl\'e within a common combinatorial framework. We then compute the cutting and co-cutting times explicitly in terms of the rotation angle and its continued-fraction expansion, making the arithmetic of the rotation directly accessible to the real dynamics.

This explicit correspondence allows us to investigate statistical properties of the class. In particular, we obtain conditions for the existence and nonexistence of absolutely continuous invariant probability measures and study the Collet--Eckmann condition. 


\subsection{Statement of the main results}
\label{subsec:statement_main_results}
In this work by a unimodal map we mean a continuous noninjective map $f\colon I\to I$ on a compact interval with a unique turning point $c\in\interior(I)$, such that $f$ is strictly monotone on each side of it; see \cref{subsec:unimodal_preliminaries}. Let $Q$ be the kneading map of $f$, and let
\[ 1=S_0<S_1<S_2<\cdots\]
be its sequence of cutting times; see \Cref{subsec:kneading_preliminaries}.

Our first result gives a kneading characterization of the long-branched Sturmian classes. In particular, it shows that within the regime $Q(k)\leq1$, the Sturmian property is equivalent to an explicit mechanical formula for the kneading map, and hence also gives an explicit formula for the cutting times.

\begin{theorem}
\label{thm:equivalence_sturmian_long_branched}
Let $f$ be a unimodal map with $Q(k) \le 1$ for every $k\geq1$. Then the following statements are equivalent:
\begin{enumerate}[(1)]
    \item The binary word $Q(1)Q(2)Q(3)\cdots$ is Sturmian.

    \item There exists a unique irrational number $\omega\in(0,1)$ such that for every $ k \ge 1$
    \begin{equation}
        \label{eq:lower_mechanical_kneading_map}
        Q(k) = \lfloor k\omega\rfloor - \lfloor(k-1)\omega\rfloor.
    \end{equation}
\end{enumerate}
In this case, $Q(1)Q(2)Q(3)\cdots$ is the lower mechanical word of slope $\omega$, and
\[ S_k = 1+k+\lfloor k\omega\rfloor = 1+\lfloor k(1+\omega)\rfloor, \]
for every $k \ge 0$.
\end{theorem}

The parameter $\omega$ in \Cref{thm:equivalence_sturmian_long_branched} is the slope of the Sturmian kneading word. For a unimodal maps as in \Cref{thm:equivalence_sturmian_long_branched} we define the \emph{rotation angle} 
\begin{equation}
    \label{eq:rotation_angle}
    \gamma = \frac{\omega}{1+\omega}.
\end{equation}
For such maps, Anu\v{s}i\'c, Bruin, and \v{C}ini\v{c} \cite{Anusic-Bruin-Cinic2020_top_properties_of_lorenz_maps} define the ``outside" rotation number
\[ \alpha \= \lim_{k \to \infty} \frac{k}{S_k}. \]
The slope of the Sturmian word $\omega$, the rotation angle $\gamma$, and the outside rotation angle $\alpha$ are related by  
\[ \alpha = 1-\gamma = \frac{1}{1+\omega}.\]

The kneading map formula in item $(2)$ in \Cref{thm:equivalence_sturmian_long_branched} characterizes the Sturmian long-branched unimodal maps with $Q(k) \le 1$. Indeed, we have that for every irrational $\omega\in(0,1)$, the function $Q(k) = \lfloor k\omega\rfloor - \lfloor(k-1)\omega\rfloor$, is an admissible kneading map, see \Cref{prop:rotation_like_Q_is_admissible}. Hence every full family of unimodal maps contains a realization of this kneading class. In the real quadratic family, its unique representative is the polynomial constructed by Bl\'e from the rigid rotation of angle $\gamma=\frac{\omega}{1+\omega}$, see \cite[Section~3]{Ble_2204_external_arguments_and_invariant_measures}. 

We say that a long-branched unimodal map $f$ is \emph{Sturmian of angle $\gamma \in (0,1/2) \setminus \Q$} if its kneading map has the form \eqref{eq:lower_mechanical_kneading_map} with $\omega$ given by the relation \eqref{eq:rotation_angle}. We denote by $\cS_{\gamma}$ the collection of all long-branched Sturmian unimodal maps of angle $\gamma$.

We next consider smooth realizations. By an $S$-unimodal map we mean a unimodal map that is $C^3$ away from its unique critical point, this critical point is non-flat of finite order $\ell>1$, and has negative Schwarzian derivative away from its critical point; see \cref{subsec:unimodal_preliminaries}. We will work with the rotation angle $\gamma$ described above. We denote by $\cS_{\gamma}^{\ell}$ the subset of $\cS_{\gamma}$ of $S$-unimodal maps with a non-flat critical point of order $\ell$. 

Let $\gamma\in(0,1/2)\setminus\Q$ with continued fraction expansion $\gamma=[a_1,a_2,a_3,\ldots]$ and convergents $\frac{p_n}{q_n}=[a_1,\ldots,a_n]$. Since $\gamma<1/2$, one has $a_1\geq2$.

The following result gives a condition for the existence of an absolutely continuous invariant probability measure.

\begin{theorem}
\label{thm:acim_rotation_like}
For every $\ell\in(1,2]$, there exists a constant $\eta_\ell>0$ such that, if
\[ \limsup_{n\to\infty} \frac{\ell^{a_{2n+1}}}{n} < \eta_\ell\]
then every $f \in \cS_{\gamma}^{\ell}$ admits an absolutely continuous invariant probability
measure.
\end{theorem}

The converse problem is governed by a different interaction between the continued-fraction coefficients and the recurrence of the critical orbit. We obtain a sufficient condition for nonexistence of an acip, valid for every finite critical order. The recurrence scale appearing in the general statement may depend on the realization.

\begin{theorem}
\label{thm:general_arithmetic_nonexistence_acip}
Let $ \ell > 1$, and let $f \in \cS_{\gamma}^{\ell}$. There exists a constant $\Lambda_f > 1$, depending on $f$, such that if
\[ \sum_{n\geq1} a_{2n+1}(q_{2n}-p_{2n}) \Lambda_f^{-q_{2n}} = +\infty,\]
then $f$ has no absolutely continuous invariant probability measure.
\end{theorem}

For the quadratic representatives the scale in the preceding theorem can be chosen uniformly; see \Cref{sec:non_existence_of_acip}. This gives a purely arithmetic nonexistence criterion for the entire quadratic family considered here and yields examples not covered by Bl\'e's earlier sufficient condition; see \Cref{ex:non_existence_beyond_ble}.

Our final result concerns hyperbolicity along the critical orbit. In contrast with the existence of an acip, which depends on the arithmetic of the angle, failure of the Collet--Eckmann condition is universal throughout the irrational long-branched Sturmian classes.

\begin{theorem}
\label{thm:no_CE_long_branched_sturmian}
Let $\gamma\in(0,1/2)\setminus\mathbb Q$ and let $\ell>1$. No map $f\in\mathcal S_\gamma^\ell$ satisfies the Collet--Eckmann condition.
\end{theorem}

These results extend and refine Bl\'e's results for quadratic maps. Bl\'e also states that every irrational quadratic representative in this class does not satisfy the Collet--Eckmann condition. His proof uses persistent recurrence of the critical point. However, a recurrent nonperiodic turning point with bounded kneading map cannot be persistently recurrent; see \cite[Theorem~6.2]{alvin-Cinic2026-Recurrence_symbolic_dynamics_and_wild_attractors}. Consequently, that argument does not establish the non-Collet--Eckmann conclusion for every irrational angle; see \cref{subsec:Ble_persistent_recurrence_gap}.

\subsection{Organization}

In \Cref{sec:preliminaries} we collect the definitions and results used throughout the paper, including the basic notions of unimodal kneading theory, cutting and co-cutting times, the principal nest and its distortion estimates, continued fractions, and Sturmian words. In \Cref{sec:sturmian_classes} we characterize the long-branched Sturmian kneading classes, prove their admissibility, identify them with the constructions of Anu\v{s}i\'c, Bruin, and \v{C}in\v{c} and of Bl\'e, and compute explicitly their cutting and co-cutting times in terms of the rotation angle and its continued-fraction convergents. In particular, we identify the quadratic representative of each class with the corresponding parameter in Bl\'e's construction.

In \Cref{sec:acim} we study the existence of absolutely continuous invariant probability measures. We  combine the combinatorial information with decay of geometry and distortion estimates to control all sufficiently deep principal-nest scaling factors. This yields the arithmetic existence criterion of \Cref{thm:acim_rotation_like}.

In \Cref{sec:non_existence_of_acip} we turn to nonexistence of absolutely continuous invariant probability measures. Using the long-branched inducing scheme, we first express existence of an acip in terms of summability of the lengths of suitable kneading neighborhoods. We then use the explicit co-cutting times to obtain lower bounds for these neighborhoods from close returns of the critical orbit, and combine them with quantitative lower bounds on critical recurrence. This to the arithmetic nonexistence criterion of \Cref{thm:general_arithmetic_nonexistence_acip} and, in the quadratic case, to a uniform version which improves Bl\'e's sufficient condition.

Finally, in \Cref{sec:non_Collet_Eckmann} we prove \Cref{thm:no_CE_long_branched_sturmian}. We first show that the Collet--Eckmann condition forces the Ces\`aro averages of the first-disagreement times of the kneading sequence to remain bounded. We then prove, using the mechanical coding of the underlying irrational rotation and the distribution of its one-sided closest returns, that these Ces\`aro averages diverge for every long-branched Sturmian kneading sequence. The resulting contradiction proves the failure of the Collet--Eckmann condition for every irrational angle. We conclude by discussing the persistent-recurrence argument used by Bl\'e and explaining why it does not establish the corresponding non-Collet--Eckmann statement.


\subsection*{Acknowledgments}
The author is grateful to Juan Rivera-Letelier for suggesting the problem that initiated this project, and to Weixiao Shen for useful discussions during the preparation of this work. This work was supported by the New Cornerstone Science Foundation through the New Cornerstone Investigator Program.


\section{Preliminaries}
\label{sec:preliminaries}

Throughout the paper, $\N$ denotes the set of positive integers and $\N_0=\N\cup\{0\}$. Every interval is equipped with the Euclidean metric and $|J|$ denotes the length of an interval $J$.

\subsection{Unimodal and \mathinhead{S}{S}-unimodal maps}
\label{subsec:unimodal_preliminaries}

In this work, a \emph{unimodal map} is a noninjective continuous map $f\colon I\to I$ on a compact interval for which there exists a unique extremum $c\in\interior(I)$, and is strictly monotone on each side of $\{c\}$. We will assume that $c$ is a maximum and thus $f$ is increasing for $x < c$ and decreasing for $x > c$. We will also assume the standard boundary condition $f(\partial I) \subset \partial I$.

Given $x \in I$ its \emph{orbit} under $f$ is the set 
\[\cO_f(x) \= \{ f^n(x) \colon n \in \N_0 \},\]
and its \emph{$\omega$-limit} set is defined as
\[ \omega_f(x) \= \{ y \in I \colon \text{ there exists } \{n_{k_i}\}_{i \ge 1} \text{ with } f^{n_{k_i}}(x) \to y \text{ as } i \to \infty \}. \]

For every $n \ge 0$ we write
\[ c_n\=f^n(c),\]
and call $\overline{\cO_f(c)}$ the \emph{postcritical} set of $f$. We restrict to the nontrivial normalization $c_2 < c < c_1$, and $c_2 < c_3$, so the core $[c_2,c_1]$ is forward invariant. Conjugation by an orientation-reversing affine homeomorphism does not change the cutting times, the kneading map, or the co-cutting times. Thus the minimum-type quadratic family $P_c(x)=x^2+c$ is covered by the same combinatorial conventions after affine conjugation.

By a $C^3$ unimodal maps, we mean a unimodal maps of class $C^1$ that is $C^3$ except at its turning point $c$. A turning point is \emph{non-flat of order $\ell>1$} if there are $C^3$ diffeomorphisms $\phi$ and $\psi$ of $\R$, with $\phi(c)=0$ and $\psi(f(c))=0$, such that
\[ |\psi(f(x))| = |\phi(x)|^\ell \]
near $c$. We call $f$ an \emph{$S$-unimodal map} if it is $C^3$ away from $c$, has a non-flat turning point, and 
\[ \frac{f'''(x)}{f'(x)} - \frac{3}{2} \left( \frac{f''(x)}{f'(x)} \right)^2 < 0, \]
wherever $f'\neq0$. All $S$-unimodal maps considered below have no attracting or parabolic periodic orbit. In particular, they have no wandering intervals and they satisfy the contraction principle; see \cite{Olivares-Vinales_Shen_non_existence_of_wandering_intervals_2025} and references therein.

A unimodal map $f$ with turning point $c$ is said to satisfy the \emph{Collet--Eckmann condition} if there exist constants $C>0$ and $\lambda>1$ such that
\[ \bigl|Df^{n}(f(c))\bigr|\ge C\lambda^n, \]
for every $n \ge 1$. Equivalently, the derivative along the forward orbit of the critical value grows exponentially.


\subsection{Itineraries, cutting times, and admissibility}
\label{subsec:kneading_preliminaries}
Let $f \colon I \longrightarrow I$ be a unimodal map with turning point $c \in I$. For $x\in I$ whose orbit does not meet $c$, its itinerary is the binary sequence $\iota(x)=\iota_0(x)\iota_1(x)\cdots$, where
\[  \iota_j(x) = 
    \begin{cases}
        0,&f^j(x)<c,\\
        1,&f^j(x)>c.
    \end{cases}
\]
If an orbit meets $c$, one can use the usual one-sided itinerary or the symbol $C$; see \cite{Milnor_Thurston1986_On_iterated_maps_of_the_interval,MevS11}. When the turning point is not periodic, its kneading sequence is the binary sequence $K(f) = e_1e_2 \ldots$ of the critical value $c_1$. Thus, $e_j = \iota_0 (c_j)$ for every $j \ge 1$. 

A binary sequence $\nu \in \{0,1\}^{\N}$ is called \emph{admissible} if it is the kneading sequence of a unimodal map with nonperiodic turning point. Let $\sigma$ denote the left shift map. In the present convention, the admissible kneading sequences are precisely the sequences $\nu=10\cdots$ such that for every $n \ge 0$
\[ \sigma(\nu) \preceq_{\rm pl} \sigma^n(\nu) \preceq_{\rm pl} \nu\]
where $\preceq_{\mathrm{pl}}$ is the parity-lexicographical order; see \cite[Theorems~3.83 and~3.90]{Bruin2022_Top_and_erg_symb_dyn}.

Define the sequence of intervals $\{D_n\}_{n \ge 1}$ inductively by $D_1 = [c,c_1]$, and for each $n \ge 2$, by 
\begin{equation}
    \label{def:Hofbauer_tower}
    D_{n} \=  
    \begin{cases}
        f(D_{n-1}) & \text{ if } c \not\in D_{n-1}, \\
        [c_{n},c_1] & \text{ if } c \in D_{n-1}.
    \end{cases} 
\end{equation}
An integer $n \ge 1$ will be called a \emph{cutting time} if $c \in D_n$. This construction was introduced by Hofbauer \cite{Hofbauer-1980_The_top_entrpy_of_axx-1} and the collection $\{D_n\}_{n \ge 1}$ is usually called the \emph{Hofbauer tower} associated with $f$.
We will denote by $\{ S_k \}_{k \ge 0}$ the sequence of cutting times. From the assumption $c_2 < c< c_1$, it follows that $S_0=1$ and $S_1 =2$. The difference of two consecutive cutting times is again a cutting time. Hence there is a map $Q\colon\N \longrightarrow \N_0$, such that for every $k \ge 1$
\begin{equation}
    \label{eq:kneading_map_recursion}
    S_k-S_{k-1} = S_{Q(k)}.
\end{equation}
This is the \emph{kneading map} of $f$.

We use the following admissibility criterion of Hofbauer and Bruin; see \cite{Bruin1995_Combinatorics_of_kneading_maps} and \cite[Theorem~3.90]{Bruin2022_Top_and_erg_symb_dyn}. For convenience, we extend $Q$ to $\N_0$ by setting $Q(0)=0$.

\begin{theorem}[Kneading-map admissibility]
    \label{thm:kneading_map_admissibility}
A map $Q\colon\N_0 \to \N_0$ with $Q(0) = 0$ is an admissible kneading map if and only if
\[ Q(k)\leq k-1 \qquad \text{and} \qquad \bigl(Q(k+j)\bigr)_{j\geq1} \succeq_{\mathrm{lex}} \bigl(Q(Q^2(k)+j)\bigr)_{j\geq1} \]
for every $k\geq1$.
\end{theorem}
In the above theorem $Q^2 = Q \circ Q$. A family of unimodal maps is \emph{full} if every admissible kneading map occurs in the family. The logistic and real quadratic families are full; see \cite{MevS11,CorRL10}.


\subsection{Co-cutting times and the co-kneading map}
\label{subsec:cocutting_preliminaries}

We recall the definition of the co-cutting times and the co-kneading map; see \cite[Section~3.6.3, in particular Exercise~3.88 and Theorem~3.90]{Bruin2022_Top_and_erg_symb_dyn}.

Let $K(f)=e_1e_2e_3\cdots$ be a nonperiodic admissible kneading sequence. For every $m\geq1$, define the \emph{first disagreement time}
\begin{equation}
    \label{eq:first_symbolic_disagreement_time}
    R(m) \= \min\{r\geq1:e_{m+r}\neq e_r\}.
\end{equation}
Since $K(f)$ is nonperiodic, $R(m)$ is finite for every $m\geq1$. Equivalently, setting
\begin{equation}
    \label{eq:rho_disagreement_relation}
    \rho(m) \= m + R(m) =\min\{n>m:e_n\neq e_{n-m}\},
\end{equation}
the words $e_{m+1}\cdots e_{\rho(m)-1}$ and $e_1\cdots e_{\rho(m)-m-1}$ coincide, while $e_{\rho(m)} \neq e_{\rho(m)-m}$.

Let$\kappa \= \min\{n\geq2:e_n=1\}$ be the position of the second occurrence of the symbol $1$. This quantity is well defined for all kneading sequences considered below. The sequence of \emph{co-cutting times} $\tS_0<\tS_1<\tS_2<\cdots$ is defined by inductively by $\tS_0=\kappa$ and for $j \ge 0$
\begin{equation}
    \label{eq:cocutting_times_disagreement_recursion}
    \tS_{j+1} \= \rho(\tS_j) = \tS_j+R(\tS_j).
\end{equation}
Equivalently, $\tS_j=\rho^j(\kappa)$, for every $j \ge 0$.

For an admissible kneading sequence, the cutting times and co-cutting times are disjoint. Moreover, the difference between two consecutive co-cutting times is a cutting time. Therefore, for every $j\geq1$, there exists a unique $\tQ(j)\in\N_0$ such that
\[ \tS_j-\tS_{j-1} = S_{\tQ(j)}.\]
The map $\tQ\colon\N\longrightarrow\N_0$ is called the \emph{co-kneading map}. Using \eqref{eq:cocutting_times_disagreement_recursion}, we obtain, $R(\tS_{j-1}) = S_{\tQ(j)}$ and hence
\[ \tS_j = \tS_{j-1}+S_{\tQ(j)}.\]


\subsection{The principal nest}
\label{subsec:principal_nest_preliminaries}

An open interval $J \subset I$ is called \emph{nice} if $f^n(\partial J) \cap J = \emptyset$ for all $n \in \N$. Given a nice interval $J$, let 
\[\cD_J \= \{ x \in I\colon \text{ there exists } k \in \N \text{ such that } f^k(x) \in J \}.\]
A connected component $J'$ of $\cD_J$ is called an \emph{entry domain} to $J$. If $J' \subset J$ we shall also call it \emph{return domain} to $J$. For $x \in \cD_J$, the minimal positive integer $k = k(x)$ with $f^k(x) \in J$ is the \emph{entry time} of $x$ to $J$. If $x \in J$ we will also call $k$ the \emph{return time} of $x$ to $J$. Note that the entry time is constant on any entry domain.

Assume that the turning point $c$ is recurrent. Let $q > c$ denote the orientation-reversing fixed point of $f$. Let $\widehat{q} < c$ with $f(\widehat{q}) = f(q)$. Choose $I_0 \= (\widehat{q},q)$, and define inductively, $I_{n+1}$ to be the return domain of $I_n$ containing $c$. This gives a sequence of nice intervals
\[ I_0\supset I_1\supset I_2\supset\cdots\ni c,\]
called the \emph{principal nest}. Let $\cR_n$ be the first return map to $I_n$. The return at level $n$ is central if $\cR_n(c)\in I_{n+1}$ and noncentral otherwise. Let
\[ m(1)<m(2)<m(3)<\cdots\]
be the noncentral moments, indexed so that the return to $I_{m(k)-1}$ is noncentral, and put
\begin{equation}
    \label{eq:central_cascade_length_acim}
    L_k \= m(k+1)-m(k).
\end{equation}
Thus the corresponding cascade contains $L_k-1$ central returns.

We will say that $f$ display \emph{decay of geometry} if there are constants $C > 0$ and $\lambda > 1$  such that for every $k \ge 1$ we have
\[ \frac{|I_{m(k)}|}{|I_{m(k)+1}|} \ge C \lambda ^k. \]

\begin{theorem}[\cite{Shen2006_decay_of_geometry}]
    \label{thm:Shen_decay_of_geometry}
    Let $f$ be a nonrenormalizable $C^3$ unimodal map with critical order $\ell \in (1,2]$. Assume that $f$ has nonperiodic recurrent critical point. Then $f$ display decay of geometry. 
\end{theorem}

The decay-of-geometry theorem is stated for an individual map. For the argument below, we also need the following uniformity consequence of its proof.

\begin{corollary}[Uniform decay rate at fixed critical order]
\label{cor:uniform_decay_fixed_critical_order}
Fix $\ell\in(1,2]$. There exists a constant $b_\ell>0$ with the following property. For every nonrenormalizable $C^3$ unimodal map $f$ of critical order $\ell$ with recurrent nonperiodic turning point, there exists an integer $k_0(f)\geq1$ such that
\begin{equation}
    \label{eq:uniform_noncentral_decay_fixed_order}
    \frac{|I_{m(k)}|}{|I_{m(k)+1}|} \ge \exp(b_\ell k)
\end{equation}
for every $k\geq k_0(f)$.
\end{corollary}

Put
\begin{equation}
    \label{eq:principal_scaling_factor_acim}
    \mu_n \= \frac{|I_{n+1}|}{|I_n|}.
\end{equation}
The following theorem was proved by Bruin, Shen, and van Strien and it is the criterion we will use for the existence of an acip; see \cite[Theorem~2]{Bruin_Shen_van_Strien2003_Inv_meas_exist_wo_growth_condition}.

\begin{theorem}
\label{thm:Bruin_shen_vs_unimodal_acip_theorem}
For every $\ell>1$, there exists $\e(\ell)>0$ with the following property. Let $f$ be an $S$-unimodal map of critical order $\ell$ with a recurrent nonperiodic turning point, and let $(I_n)$ be its principal nest. If
\[ \frac{|I_{n+1}|}{|I_n|} \leq \e(\ell)\]
for every sufficiently large $n$, then $f$ admits an acip.
\end{theorem}


\subsection{Principal-nest extensions and distortion}
\label{subsec:principal_nest_distortion}

We recall the pullback-extension, real-bounds, and distortion estimates used below. The extension is given in \cite[Corollary~2.3]{Shen2006_decay_of_geometry}, the required real bounds originate in the work of Martens \cite{Martens1994_Distortion_results_and_inv_Cantor_sets}, and the cross-ratio estimate in Kozlovski's \cite[Theorem~C]{Kozlovski2000_Getting_rid_of_Sd}.

Let $J\Subset T$ be intervals, and let $L$ and $R$ be the components of $T\setminus J$. We say that $T$ contains a \emph{$\tau$-scaled neighborhood} of $J$ if $|L|\geq\tau|J|$ and $|R|\geq\tau|J|$. 
Define
\[ \Cr(T,J) \= \frac{|T||J|}{|L||R|}.\]
If $h\colon T\to h(T)$ is a diffeomorphism, its \emph{cross-ratio distortion} on the pair $(T,J)$ is
\[ \Cr(h;T,J) \= \frac{\Cr(h(T),h(J))} {\Cr(T,J)}.\]

We shall use the following form of the real Koebe principle.

\begin{lemma}[Real Koebe principle]
    \label{lem:real_Koebe_principal_nest}
Let $\tau>0$ and $C\in(0,1]$ be constants. Let $J \Subset T$ be intervals, and let $h\colon T\to h(T)$ be a diffeomorphism. Suppose that
\[ \Cr(h;T',J') \geq C\]
for every pair of intervals $J'\Subset T'\subset T$. If $h(T)$ contains a $\tau$-scaled neighborhood of $h(J)$, then
\[ \frac{|h'(x)|}{|h'(y)|} \le C^{-6}\frac{(1+\tau)^2}{\tau^2}\]
for every $x,y\in J$.
\end{lemma}

For a $C^1$ diffeomorphism $h\colon J\to h(J)$, we define its \emph{derivative distortion} on $J$ by
\[ \Dist(h|_J) \= \sup_{x,y\in J} \frac{|h'(x)|}{|h'(y)|}.\]
We say that $h$ has distortion bounded by $D\geq1$ on $J$ if
\[ \Dist(h|_J)\leq D.\]
Whenever the composition is defined, derivative distortion satisfies
\[ \Dist(g\circ h|_J) \leq \Dist(g|_{h(J)}) \Dist(h|_J).\]

For later use, we record the resulting bounded-distortion statement for the principal nest. Let $s_k$ be the first return time of the turning point $c$ to $I_{m(k)}$, and put $H_k \= f^{s_k-1}$.

\begin{proposition}
\label{prop:principal_nest_extension_distortion}
Fix $\ell\in(1,2]$. There exists a constant $D_\ell\geq1$ with the following property. For every nonrenormalizable $C^3$ unimodal map $f$ of critical order $\ell$ with recurrent nonperiodic turning point, there exists $k_*(f)\geq1$ such that, for every $k\geq k_*(f)$, the following holds. There exists an interval $U_k\supset f(I_{m(k)+1})$ such that $H_k\colon U_k\longrightarrow I_{m(k)-1}$ is a diffeomorphism and
\begin{equation}
    \label{eq:uniform_principal_nest_distortion}
    \sup_{x,y\in f(I_{m(k)+1})} \frac{|H_k'(x)|}{|H_k'(y)|} \leq D_\ell.
\end{equation}
\end{proposition}




\subsection{Continued fractions}
\label{subsec:continued_fraction_preliminaries}

For $\gamma\in(0,1/2)\setminus\Q$, write
\[ \gamma=[a_1,a_2,a_3,\ldots] = \cfrac{1}{a_1+\cfrac{1}{a_2+\cfrac{1}{a_3+\cdots}}}, \]
for its continued fraction expansion, and
\[ \frac{p_n}{q_n}=[a_1,\ldots,a_n]\]
for its convergents, with
\[ p_0=0, \quad q_0=1, \quad p_1=1, \quad q_1=a_1.\]
Then
\begin{equation}
    \label{eq:continued_fraction_recursion}
    p_{n+1}=a_{n+1}p_n+p_{n-1}, \qquad q_{n+1}=a_{n+1}q_n+q_{n-1},
\end{equation}
and 
\begin{equation}
    \label{eq:best_approx_convergents}
    \frac{p_{2n}}{q_{2n}} < \gamma < \frac{p_{2n+1}}{q_{2n+1}},
\end{equation}
for every $n \ge0$.


\subsection{Sturmian words}
\label{subsec:sturmian_words_preliminaries}

Let $u=u_1u_2u_3\cdots \in \{0,1\}^{\N}$ be an infinite binary word. For $n \in \N$, We denote by $\mathcal L_n(u)$ the collection of its subwords of length $n$ and by
\[ p_u(n) \= \#\mathcal L_n(u)\]
its word-complexity. For a finite binary word $v$, let $|v|_1$ denote the number of symbols $1$ occurring in $v$.

An infinite binary word $u$ is called \emph{Sturmian} if $p_u(n)=n+1$ for every $n\geq1$. The word $u$ is called \emph{balanced} if $\bigl||v|_1-|w|_1\bigr| \leq 1$ for every $n\geq1$ and all $v,w\in\mathcal L_n(u)$.

For an irrational number $\omega\in(0,1)$, the \emph{lower mechanical word} of slope $\omega$ is the binary word
\begin{equation}
    \label{eq:lower_mechanical_word}
    \underline{s}_{\omega}(k) \= \lfloor k\omega\rfloor - \lfloor(k-1)\omega\rfloor, \qquad k\geq1.
\end{equation}

We recall the following standard characterizations; see \cite[Definition~4.60, Lemmas~4.58 and~4.63, Theorem~4.66, Proposition~4.67, and Theorem~4.69]{Bruin2022_Top_and_erg_symb_dyn}.

\begin{proposition}[Sturmian characterizations]
\label{prop:sturmian_characterizations}
Let $u$ be an infinite binary word. The following statements are equivalent:
\begin{enumerate}[(1)]
    \item $u$ is Sturmian;
    \item $u$ is aperiodic and balanced;
    \item there exist an irrational number $\omega\in(0,1)$ and some $\beta\in\R$ such that $u$ is given by one of the mechanical codings
    \[ u_k = \lfloor k\omega+\beta\rfloor - \lfloor(k-1)\omega+\beta\rfloor \]
    or
    \[ u_k = \lceil k\omega+\beta\rceil - \lceil(k-1)\omega+\beta\rceil.\]
\end{enumerate}
Moreover, the slope $\omega$ is unique.
\end{proposition}

The following elementary consequence will be used below. We denote by $\preceq_{\rm lex}$ the lexicographical order.

\begin{lemma}
    \label{lem:lexicographically_minimal_sturmian}
Let $u$ be a Sturmian word of slope $\omega$ satisfying $u \preceq_{\rm lex} \sigma^m(u)$ for every $m\geq0$. Then $ |u_1u_2\cdots u_n|_1 = \lfloor n\omega\rfloor$ for every $n\geq1$. Consequently,
\[ u_n = \lfloor n\omega\rfloor - \lfloor(n-1)\omega\rfloor,\] 
for every $n\geq1$.
\end{lemma}

\begin{proof}
For every $m\geq0$, the hypothesis $u\preceq_{\mathrm{lex}}\sigma^m(u)$ implies that the prefix $u_1\cdots u_n$ is lexicographically no larger than the factor $u_{m+1}\cdots u_{m+n}$. Hence the prefix is lexicographically minimal among all factors of $u$ of length $n$.

Every factor of a Sturmian word of slope $\omega$ and length $n$ contains either $\lfloor n\omega\rfloor$ or $\lceil n\omega\rceil$ symbols $1$, and both possibilities occur. Suppose that the prefix $u_1\cdots u_n$ contains $\lceil n\omega\rceil$ symbols $1$, and choose a factor $v=v_1\cdots v_n$ containing $\lfloor n\omega\rfloor$ symbols $1$. Let $r$ be the first position at which these two words differ. Lexicographic minimality gives $u_r=0$, and $v_r=1$. The two words agree before position $r$, while the prefix contains one more symbol $1$ in total. It follows that
\[ |u_{r+1}\cdots u_n|_1 - |v_{r+1}\cdots v_n|_1 =2.\]
The two suffixes are factors of $u$ of the same length, contradicting balance. Therefore 
\[ |u_1\cdots u_n|_1=\lfloor n\omega\rfloor.\] 
Taking consecutive differences gives
\[ u_n = \lfloor n\omega\rfloor-\lfloor(n-1)\omega\rfloor.\]
\end{proof}


\section{Long-branched Sturmian kneading classes}
\label{sec:sturmian_classes}

Fix $\gamma\in(0,1/2)\setminus\Q$ and put
\[ \omega_\gamma \= \frac{\gamma}{1-\gamma}.\]
Define $Q_\gamma(0)=0$ and, for $k\geq1$,
\[ Q_\gamma(k) = \lfloor k\omega_\gamma\rfloor - \lfloor(k-1)\omega_\gamma\rfloor. \]
Observe that, since $\omega_{\gamma} \in (0,1)$, we have $Q_{\gamma}(k) \in \{0,1\}$ for every $k \in \N$.

\begin{proposition}
\label{prop:rotation_like_Q_is_admissible}
For every irrational $\gamma\in(0,1/2)$, the map $Q_\gamma$ is admissible. 
\end{proposition}

\begin{proof}
Let $\gamma \in (0,1/2) \setminus \Q$. Since $0<\omega_\gamma<1$, we have $Q_{\gamma}(1) = 0$, and $Q_\gamma(k)\in\{0,1\}$ for every $k \ge 1$. Thus $Q_\gamma(k)\leq k-1$ for every $k\geq1$. Moreover, $Q_\gamma^2(k)=0$ for every $k\geq1$, because $Q_\gamma(0)=Q_\gamma(1)=0$.

Now, fix $m,n\geq1$. Telescoping gives
\[ \sum_{j=1}^{n}Q_\gamma(m+j) = \lfloor(m+n)\omega_\gamma\rfloor - \lfloor m\omega_\gamma\rfloor \geq \lfloor n\omega_\gamma\rfloor = \sum_{j=1}^{n}Q_\gamma(j). \]
If the two words coincide, the required lexicographical inequality is immediate. If the words $(Q_\gamma(m+j))_{j\geq1}$ and $(Q_\gamma(j))_{j\geq1}$ first differ at $j =r$, their partial sums agree up to $r-1$, while the preceding inequality forces the shifted word to have symbol $1$ and the original word symbol $0$ at position $r$. Hence
\[  \bigl(Q_\gamma(m+j)\bigr)_{j\geq1} \succeq_{\rm pl} \bigl(Q_\gamma(j)\bigr)_{j\geq1} =  \bigl(Q_\gamma(Q_\gamma^2(m)+j)\bigr)_{j\geq1},\]
 where the last equality comes from $Q_\gamma^2(m)=0$. The conclusion follows from \Cref{thm:kneading_map_admissibility}.
\end{proof}

Recall that the collection of all unimodal maps with kneading map $Q_\gamma$ is denoted by $\cS_\gamma$, and $\cS_{\gamma}^{\ell}$ denote the subclass of $\cS_{\gamma}$ consisting of $S$-unimodal maps of critical order $\ell$. By \Cref{prop:rotation_like_Q_is_admissible}, $\cS_\gamma$ has a representative in every full family of unimodal maps.

The cutting times and the kneading sequence follow directly from the definition. Recall that that the turning point is assumed to be a maximum.

\begin{proposition}
\label{prop:cutting_times_kneading_sequence_sturmian}
Let $\gamma \in (0,1/2) \setminus \Q$, and let $f\in\cS_\gamma$. Then
\begin{equation}
    \label{eq:explicit_cutting_times_rotation_like}
    S_k = 1+k+\lfloor k\omega_\gamma\rfloor = 1+\left\lfloor\frac{k}{1-\gamma}\right\rfloor, \qquad k\geq0.
\end{equation}
\end{proposition}

\begin{proof}
Since $Q_\gamma(k)\in\{0,1\}$ and $S_0=1$, one has $S_{Q_\gamma(k)}=1+Q_\gamma(k)$. The recursion \cref{eq:kneading_map_recursion} therefore gives
\[ S_k = 1+\sum_{j=1}^{k}\bigl(1+Q_\gamma(j)\bigr) = 1+k+\lfloor k\omega_\gamma\rfloor = 1+\left\lfloor\frac{k}{1-\gamma}\right\rfloor.\]
\end{proof}

\begin{remark}
A way to obtain the kneading sequence from $Q_\gamma$ is by the variable-length coding
\begin{equation}
    \label{eq:block_morphism_Q_to_nu_maximum}
    0 \longmapsto 0, \qquad 1 \longmapsto 11,
\end{equation}
preceded by the initial symbol $1$. If we assume that the turning point of $f$ is a minimum, the corresponding coding is
\[ 0 \longmapsto 1, \qquad 1 \longmapsto 00,\]
preceded by $0$; see \cite[Section~5 and the proof of Proposition~2]{Anusic-Bruin-Cinic2020_top_properties_of_lorenz_maps}. 
\end{remark}

The word $Q_\gamma(1)Q_\gamma(2)\cdots$ is the lower mechanical word of slope $\omega_\gamma$ and, by \Cref{prop:sturmian_characterizations}, is therefore Sturmian. Moreover, if $\gamma=[a_1,a_2,a_3,\ldots]$, then
\begin{equation}
    \label{eq:continued_fraction_kneading_slope}
    \omega_\gamma = [a_1-1,a_2,a_3,\ldots].
\end{equation}
In particular, $Q_\gamma$ is bounded and every map in $\cS_\gamma$ is long-branched. Also,
\begin{equation}
    \label{eq:outside_rotation_limit}
    \lim_{k\to\infty}\frac{k}{S_k} = 1-\gamma.
\end{equation}
The three natural parameters are therefore related by
\begin{equation}
    \label{eq:parameter_dictionary_sturmian}
    \alpha=1-\gamma, \qquad \omega_\gamma=\frac{1-\alpha}{\alpha}, \qquad \gamma=\frac{\omega_\gamma}{1+\omega_\gamma}.
\end{equation}

We now prove \Cref{thm:equivalence_sturmian_long_branched}.

\begin{proof}[Proof of \Cref{thm:equivalence_sturmian_long_branched}]
Let $f$ be a unimodal map as in the statement of the theorem. Put
\[ w \= Q(1)Q(2)Q(3)\cdots.\]
Since $Q(k)\leq1$ and $Q$ takes values in $\N_0$, one has $Q(k)\in\{0,1\}$ for every $k\geq1$. Moreover, $Q(1)=0$, and hence $Q^2(k)=0$ for every $k\geq1$. The kneading-map admissibility condition therefore gives, for every $m\geq1$,
\[ \sigma^m(w) = \bigl(Q(m+j)\bigr)_{j\geq1} \succeq_{\mathrm{lex}} \bigl(Q(Q^2(m)+j)\bigr)_{j\geq1} = w.\]
Thus $w$ is lexicographically smaller than or equal to each of its shifts.

Suppose first that item~$(2)$ holds. Then $w$ is the lower mechanical word of the irrational slope $\omega$, and hence it is Sturmian. This proves item~$(1)$.

Conversely, suppose that $w$ is Sturmian. The kneading-map admissibility criterion therefore gives
\[ w\preceq_{\mathrm{lex}}\sigma^m(w)\]
for every $m\geq0$. By \Cref{lem:lexicographically_minimal_sturmian}, there is a unique irrational $\omega\in(0,1)$ such that
\[ Q(k) = \lfloor k\omega\rfloor-\lfloor(k-1)\omega\rfloor\]
for every $k\geq1$.

Finally, since $Q(j)\in\{0,1\}$, one has $S_{Q(j)}=1+Q(j)$. Hence
\[ S_k = 1+\sum_{j=1}^{k}S_{Q(j)} =1+k+\sum_{j=1}^{k}Q(j) = 1+k+\lfloor k\omega\rfloor.\]
Since $k$ is an integer, this is also
\[ S_k =1+\lfloor k(1+\omega)\rfloor,\]
as required.
\end{proof}

\begin{remark}
The binary word $Q_\gamma(1)Q_\gamma(2)\cdots$ is Sturmian, but it should not be confused with the full symbolic dynamics of the unimodal map. By \Cref{prop:cutting_times_kneading_sequence_sturmian}, the kneading sequence is reconstructed from this word by a variable-length coding. Moreover, the postcritical restriction is represented by the corresponding one-sided Sturmian shift; see \Cref{prop:Alvin_Cinc_outside_rotation}. The full unimodal shift, however, need not be Sturmian and may have positive entropy.
\end{remark}

Anu\v{s}i\'c, Bruin, and \v{C}in\v{c} associate to a symmetric unimodal map an increasing Lorenz map and, after introducing a plateau, a nondecreasing degree-one circle map. In the irrational case, the rotation number $\alpha\in(1/2,1)$ of this circle map determines a long-branched Sturmian kneading class; see \cite{Anusic-Bruin-Cinic2020_top_properties_of_lorenz_maps}, and \cite[Section~6.1 and Proposition~6.6]{alvin-Cinic2026-Recurrence_symbolic_dynamics_and_wild_attractors}. We call $\alpha$ the \emph{outside rotation number}. The following is proved in \cite[Section~6.1 and Proposition~6.6]{alvin-Cinic2026-Recurrence_symbolic_dynamics_and_wild_attractors}.

\begin{proposition}[Alvin--\v{C}in\v{c}]
\label{prop:Alvin_Cinc_outside_rotation}
Let $f$ belong to an irrational Sturmian class of outside rotation number $\alpha$. Then $Q(k) \le 1$ for every $k \ge 1$, and
\[ \alpha = \lim_{k\to\infty}\frac{k}{S_k}.\]
Moreover,
\[ f\colon\omega_f(c)\longrightarrow\omega_f(c) \]
is a noninvertible minimal surjection represented by the Sturmian shift of rotation number $\alpha$. Every irrational $\alpha\in(1/2,1)$ is realized by this construction.
\end{proposition}

The next proposition identifies their geometrically defined class with our intrinsic kneading class: the parameters are related by $\alpha=1-\gamma$, and the equivalence follows by comparing the corresponding cutting times and kneading maps.

\begin{proposition}
\label{prop:equivalence_Alvin_Cinc}
For every irrational $\gamma\in(0,1/2)$, the kneading class $\cS_\gamma$ coincides with the irrational long-branched Sturmian class of Anu\v{s}i\'c, Bruin, and \v{C}in\v{c} having outside rotation number $\alpha=1-\gamma$.
\end{proposition}

\begin{proof}
Fix an irrational $\gamma\in(0,1/2)$ and put
\[ \alpha=1-\gamma, \qquad \text{and} \qquad \omega=\frac{1-\alpha}{\alpha} =\frac{\gamma}{1-\gamma}.\]
By \Cref{prop:Alvin_Cinc_outside_rotation}, there exists an irrational
Sturmian class of outside rotation number $\alpha$.

For every $n \ge 1$, consider the characteristic coding of the rotation $R_\alpha$, given by
\[ u_n = \lceil n\alpha\rceil - \lceil(n-1)\alpha\rceil.\]
The cutting time $S_k$ is the position of the $(k+1)$-st occurrence of the symbol $1$ in this coding. Since
\[ \sum_{j=1}^{n}u_j = \lceil n\alpha\rceil,\]
they are given by $S_0 =1$, and for every $k \ge 1$
\[ S_k = \left\lceil\frac{k}{\alpha}\right\rceil = 1+\left\lfloor\frac{k}{\alpha}\right\rfloor.\]
By \Cref{prop:Alvin_Cinc_outside_rotation}, the kneading map takes only the values $0$ and $1$. Hence $S_{Q(k)} = 1+Q(k)$, and the kneading recursion gives $Q(k) =S_k-S_{k-1}-1$. Using $1/\alpha=1+\omega$, we obtain
\[ Q(k) = \left\lfloor\frac{k}{\alpha}\right\rfloor - \left\lfloor\frac{k-1}{\alpha}\right\rfloor -1 = \lfloor k\omega\rfloor - \lfloor(k-1)\omega\rfloor = Q_\gamma(k). \]
Thus the class of outside rotation number $\alpha=1-\gamma$ has kneading map $Q_\gamma$, and therefore defines the same kneading class $\cS_\gamma$.
\end{proof}

As an immediate consequence we obtain the following; see \cite[Proposition~6.7]{alvin-Cinic2026-Recurrence_symbolic_dynamics_and_wild_attractors}

\begin{corollary}
\label{cor:uniform_recurrence_long_branched_sturmian}
Let $f\in\cS_\gamma$ satisfy the contraction principle. Then the turning point of $f$ is nonperiodic and uniformly recurrent, and $\tQ(j)\longrightarrow+\infty$ as $j \to +\infty$.
\end{corollary}

We also note that every realization in this class without an essential periodic attractor is nonrenormalizable. Indeed, Bruin's renormalization criterion \cite[Proposition~1]{Bruin1995_Combinatorics_of_kneading_maps} states that, in the absence of an essential periodic attractor, renormalizability with period $S_k$ is equivalent to $ Q(k+j) \geq k$ for every $j\geq1$. This is impossible for $Q = Q_\gamma$. For $k\geq2$, it follows immediately from $Q_\gamma\leq1$, while for $k=1$ it follows from the fact that the irrational mechanical word $Q_\gamma(1)Q_\gamma(2)\cdots$ contains infinitely many zeros.


\subsection{Co-cutting times and convergent returns}
\label{subsec:co_cutting_convergent_returns}

In \cite[Proposition~6.7]{alvin-Cinic2026-Recurrence_symbolic_dynamics_and_wild_attractors} it was proved that the co-kneading map associated with a map $f \in \cS_{\gamma}$ diverges, and the turning point is uniformly recurren. Here we will compute the co-cutting times in term of the angle $\gamma$. Recall that 
\[ \gamma = [a_1, a_2, \ldots] \qquad \text{and} \qquad \frac{p_n}{q_n}= [a_1, \ldots, a_n], \]
are the continued fraction expansion and the $n$-th convergent of $\gamma$ respectively.

Put $A_0 = 0$, and for $m \ge 1$ put
\[ A_m=\sum_{r=1}^{m}a_{2r}.\]
Thus $A_m$ is the sum of the first $m$ even-indexed partial quotients of $\gamma$. We enumerate the co-cutting times as $\tS_0<\tS_1<\cdots$.

Under the correspondence with the Sturmian construction of Anu\v{s}i\'c, Bruin and \v{C}in\v{c}, the co-cutting times are precisely the lower one-sided record-return times of the rigid rotation $R_\gamma$; see \cite[Section~6.1]{alvin-Cinic2026-Recurrence_symbolic_dynamics_and_wild_attractors}.

\begin{theorem}
\label{thm:explicit_co_cutting_times_rotation_like}
Let $f\in\cS_\gamma$. For every $m\geq0$ and $1\leq j\leq a_{2m+2}$, we have
\begin{equation}
    \label{eq:explicit_co_cutting_times_rotation_like}
    \tS_{A_m+j-1} = q_{2m}+j q_{2m+1}.
\end{equation}
In particular,  
\begin{equation}
    \label{eq:terminal_and_next_co_cutting_times}
    \tS_{A_m-1}=q_{2m}, \qquad \tS_{A_m}=q_{2m}+q_{2m+1},
\end{equation}
for every $m\geq1$. Moreover, if $d_m \= q_{2m+1}-p_{2m+1}$, then
\begin{equation}
    \label{eq:odd_q_as_cutting_time}
        S_{d_m}=q_{2m+1}
\end{equation}
for every $m\geq0$, and
\begin{equation}
    \label{eq:disagreement_time_even_convergent}
        R(q_{2m})=q_{2m+1}
\end{equation}
for every $m\geq1$.
\end{theorem}

\begin{proof}
For $m\geq0$ and $1\leq j\leq a_{2m+2}$, put
\[ N_{m,j} \= q_{2m}+j q_{2m+1}, \qquad \text{and} \qquad P_{m,j} \= p_{2m}+j p_{2m+1}.\]
Set
\[ \delta_{2m} \= q_{2m}\gamma-p_{2m}>0, \qquad \text{and} \qquad \delta_{2m+1} \= p_{2m+1}-q_{2m+1}\gamma>0.\]
The continued-fraction recursion gives $\delta_{2m} = a_{2m+2}\delta_{2m+1} + \delta_{2m+2}$. Consequently,
\[ N_{m,j}\gamma-P_{m,j} = \delta_{2m}-j\delta_{2m+1} = (a_{2m+2}-j)\delta_{2m+1} + \delta_{2m+2} >0.\]
Thus $P_{m,j}/N_{m,j}<\gamma$, and these approximation errors decrease
strictly as $j$ increases.

For every $1\leq j\leq a_{2m+2}$, the fractions
\[ \frac{P_{m,j}}{N_{m,j}} = \frac{p_{2m}+j p_{2m+1}} {q_{2m}+j q_{2m+1}}, \]
are precisely the intermediate, or Farey, convergents to $\gamma$ from below. By the best-approximation property of the Farey convergents, they give, in increasing order of their denominators, the successive lower one-sided record returns of $R_\gamma$.
By the co-cutting correspondence recalled above, these denominators are therefore the successive co-cutting times.

The block indexed by $m$ contains $a_{2m+2}$ terms. Since the number of terms in the preceding blocks is
\[ \sum_{r=0}^{m-1}a_{2r+2} = A_m,\]
the $j$-th term of the $m$-th block has index $A_m+j-1$. Hence \cref{eq:explicit_co_cutting_times_rotation_like} holds.

Taking $j=a_{2m+2}$ and using the continued-fraction recursion gives
\[ q_{2m}+a_{2m+2}q_{2m+1} =q_{2m+2}.\]
Taking $j=a_{2m+2}$ gives $\widetilde S_{A_{m+1}-1} = q_{2m+2}$. Replacing $m+1$ by $m$ yields $\widetilde S_{A_m-1}=q_{2m}$. Taking $j=1$ in the block indexed by $m$ gives $\widetilde S_{A_m}=q_{2m}+q_{2m+1}$.

We next prove \cref{eq:odd_q_as_cutting_time}. Put $\varepsilon_{2m+1} \= p_{2m+1}-q_{2m+1}\gamma$.
The convergent estimate gives
\[ 0<\varepsilon_{2m+1} < \frac{1}{q_{2m+2}} < 1-\gamma.\]
Since $d_m=q_{2m+1}-p_{2m+1}$, it follows that
\[ \frac{d_m}{1-\gamma} = q_{2m+1} - \frac{\varepsilon_{2m+1}}{1-\gamma} \in (q_{2m+1}-1,q_{2m+1}).\]
Using the cutting-time formula
\[ S_k = 1+\left\lfloor\frac{k}{1-\gamma}\right\rfloor\]
with $k=d_m$, we obtain $S_{d_m}=q_{2m+1}$. 

Finally, if $m\geq1$, then $q_{2m}$ is a co-cutting time and
\cref{eq:terminal_and_next_co_cutting_times} shows that the next one is
$q_{2m}+q_{2m+1}$. Since consecutive co-cutting times satisfy $\tS_{j+1} = \tS_j+R(\tS_j)$, we conclude \eqref{eq:disagreement_time_even_convergent}.
\end{proof}

\begin{example}[The golden-mean case]
Let
\[ \gamma_* = \frac{3-\sqrt5}{2} = [2,1,1,\ldots], \qquad \text{and} \qquad \varphi = \frac{1+\sqrt5}{2},\]
and let $\{F_k\}_{k \ge0}$ be the sequence of Fibonacci number with $F_0 = 0$, and $F_1 = 1$.
For the long-branched Sturmian class of angle $\gamma_*$, the cutting times
are given by
\[ S_k = 1+\lfloor k\varphi\rfloor,\]
and hence $S_k\sim\varphi k$. Thus the cutting times grow linearly. By contrast, \Cref{thm:explicit_co_cutting_times_rotation_like} shows that the co-cutting times corresponding to the deepest returns are $\tS_{m-1} = q_{2m} = F_{2m+2}$, for every $m \ge 1$. Consequently,
\[ \tS_{m-1} \sim \frac{\varphi^{2m+2}}{\sqrt5},\]
so these deep returns occur along an exponentially growing Fibonacci subsequence.

For the Fibonacci map, on the other hand, we have $Q_{\rm Fib}(k) = \max\{k-2,0\}$, and the cutting times themselves satisfy 
\[ S_k = F_{k+2} \sim \frac{\varphi^{k+2}}{\sqrt5};\]
see \cite{Lyubich-Milnor1993_Fib_unimodal} and \cite[Example~3.84]{Bruin2022_Top_and_erg_symb_dyn}. Hence the cutting times grow exponentially in the Fibonacci case, whereas they grow linearly in the long-branched Sturmian case. The Fibonacci structure occurs at different combinatorial levels, in the former it appears in the cutting times themselves, while in the latter it appears in the co-cutting times and the continued-fraction denominators.

Despite this difference, both postcritical systems are semiconjugate to the golden-mean rotation. Equivalently, in the convention used for the present class one may use $R_{\gamma_*}$, since the rotations of angles $\gamma_*$ and $1-\gamma_*$ are conjugate by orientation reversal. Thus the two maps have the same irrational rotation factor but are not kneading-equivalent, the present kneading map is bounded and the map is long-branched, whereas the Fibonacci kneading map diverges.
\end{example}


\subsection{The quadratic representative and Bl\'e's construction}
\label{subsec:quadratic_realization_Ble}

We now identify the quadratic representatives of the long-branched Sturmian classes with the real quadratic maps constructed by Bl\'e in \cite{Ble_2204_external_arguments_and_invariant_measures}. More precisely, for every irrational $\gamma\in(0,1/2)$, we will prove that the map constructed by Bl\'e from the rigid rotation of angle $\gamma$ is exactly the representative of $\cS_\gamma$ in the real quadratic family $P_t(x)=x^2+t$. Thus Bl\'e's construction gives a geometric realization of the quadratic members of the intrinsic kneading classes considered here.

There is a minor difference of conventions that must first be made explicit. Throughout this paper our unimodal maps have a maximum turning point, whereas $P_t$ has a minimum at $0$. Passing from one convention to the other by an orientation-reversing affine conjugacy exchanges the two sides of the turning point. Consequently, the binary kneading sequence is complemented ($0$'s should be read as $1$'s and $1$'s as $0$'s), while the cutting times and the kneading map remain unchanged.

In particular, if $\nu_\gamma^{\max}$ denotes the kneading sequence associated with $Q_\gamma$ in our maximum convention, then
\[ \nu_\gamma^{\max} = 1\, \Phi_{\max}(Q_\gamma(1)) \Phi_{\max}(Q_\gamma(2))\cdots,\]
where
\[ \Phi_{\max}(0)=0, \qquad \text{and} \qquad \Phi_{\max}(1)=11.\]
For a minimum turning point the corresponding sequence is
\begin{equation}
    \label{eq:min_kneading_sequence_Q_gamma}
    \nu_\gamma^{\min} = 0\, \Phi_{\min}(Q_\gamma(1)) \Phi_{\min}(Q_\gamma(2))\cdots,
\end{equation}
where
\begin{equation}
    \label{eq:min_block_coding_Q_gamma}
    \Phi_{\min}(0)=1, \qquad \text{and} \qquad \Phi_{\min}(1)=00.
\end{equation}
Thus $\nu_\gamma^{\max}$ and $\nu_\gamma^{\min}$ determine the same cutting times and the same kneading map $Q_\gamma$.

The admissibility of $Q_\gamma$ and fullness of the real quadratic family give a quadratic realization. Since the corresponding kneading sequence is aperiodic, real quadratic rigidity implies uniqueness; see \cite{Kozlovski-She_van_Strien2007_rigidity_for_real_polynomials}. We denote this parameter by $t_\gamma\in[-2,-1]$. Thus the maximum-type affine conjugate of $P_{t_\gamma}$ belongs to $\cS_\gamma$.

Our identification with Bl\'e's construction will be made through external angles. We will compute the characteristic external angle of the quadratic representative $P_{t_\gamma}$ directly from $Q_\gamma$, and show that it coincides with the angle used by Bl\'e; see \cite[Section~3.2]{Ble_2204_external_arguments_and_invariant_measures}. We use throughout the standard correspondence between real quadratic kneading data and external angles; see \cite{Douady1986_algorithm_for_computing_external_angles} for the general construction and the formulas underlying the computations below.

We first recall the kneading-to-external-angle correspondence in the form that we need. For convenience, write $ \nu_\gamma^{\min} = e_1 e_2 e_3\cdots$, where $e_1=0$. For every $j\geq0$, let $\varepsilon_j\in\{-1,+1\}$ be the orientation of the lap containing the $j$th iterate of the critical orbit, thus
\begin{equation}
    \label{eq:lap_orientation_sign}
    \varepsilon_j =
    \begin{cases}
        -1, \text{ if } e_j=0,\\
        +1, \text{ if } e_j=1.
    \end{cases}
\end{equation}
Indeed, $P_t$ is decreasing on the left of $0$ and increasing on the right of $0$. Hence
\[ \Pi_n = \prod_{j=1}^{n}\varepsilon_j\]
is the orientation of the monotonicity branch of $P_t^n$ following the first $n$ symbols of the critical itinerary. Thus, $\Pi_n=+1$ for an increasing branch and $\Pi_n=-1$ for a decreasing branch.

Choose $b_1=1$ and define $b_{n}\in\{0,1\}$, for $n > 1$, by $\Pi_n=(-1)^{b_{n+1}+b_1}$. Equivalently,
\begin{equation}
    \label{eq:consecutive_angle_digits_kneading} 
    \varepsilon_n = (-1)^{b_n+b_{n+1}}.
\end{equation}
The corresponding external angle in base $2$ is 
\[  \vartheta_\gamma = \sum_{n\geq1}\frac{b_n}{2^{n}} = 0.b_1b_2b_3\cdots.\]
The other choice $b_1=0$ gives the conjugate external angle $1-\vartheta_\gamma$. These two angles determine the same real kneading data.

We next recall the external angle appearing in Bl\'e's construction.Let $\theta_\gamma$ be the external argument of the point of the main cardioid with irrational internal argument $\gamma$, and let 
\[s_\gamma(n) = \lfloor(n+1)\gamma\rfloor-\lfloor n\gamma\rfloor.\]
The staircase coding of the rotation of angle $\gamma$ gives
\begin{equation}
    \label{eq:characteristic_external_angle_rotation}
    \theta_\gamma = \sum_{n\geq1} \frac{s_\gamma(n)}{2^n},
\end{equation}
Equivalently, $\theta_\gamma = 0.s_\gamma(1)s_\gamma(2)s_\gamma(3)\cdots$, in base $2$; see \cite{Bullett-Sentenac1994_Ordered_orbits_of_the_shift}. This is the characteristic Sturmian coding of the rigid rotation. Bl\'e then considers the map $T(\theta) = \frac12+\frac{\theta}{4}$.

The relation between these two angle constructions is the following.

\begin{lemma}
\label{lem:external_angle_Q_gamma}
For every irrational $\gamma\in(0,1/2)$, the external angle $\vartheta_\gamma$ obtained from the minimum-type kneading sequence $\nu_\gamma^{\min}$ satisfies
\[ \vartheta_\gamma = T(\theta_\gamma) = \frac12+\frac{\theta_\gamma}{4}.\]
\end{lemma}

\begin{proof}
Put $\omega_\gamma = \gamma/ (1-\gamma)$. We first determine the binary expansion of $\vartheta_\gamma$ directly from the block structure \eqref{eq:min_kneading_sequence_Q_gamma}.

Since $e_1=0$, one has $\varepsilon_1=-1$. As $b_1=1$, \eqref{eq:consecutive_angle_digits_kneading} gives $b_2=0$. Hence the binary expansion of $\vartheta_\gamma$ begins with $10$.

After this initial symbol the orientation state is $b_2=0$. A block $\Phi_{\min}(0)=1$ has sign $+1$ and therefore leaves the binary state unchanged, so it contributes the digit $0$. A block $\Phi_{\min}(1)=00$ has signs $(-1,-1)$ and changes the state $0\to1\to0$, so it contributes the two digits $10$.

Define $H(0)=0$, and $H(1)=10$. It follows that
\begin{equation}
    \label{eq:external_angle_Q_morphism}
    \vartheta_\gamma = 0.10\, H(Q_\gamma(1)) H(Q_\gamma(2)) H(Q_\gamma(3))\cdots.
\end{equation}

We claim that
\begin{equation}
    \label{eq:H_Q_characteristic_rotation_word}
    H(Q_\gamma(1))H(Q_\gamma(2))\cdots = s_\gamma(1)s_\gamma(2)\cdots.
\end{equation}
To prove this, let $k_m$ be the position of the $m$-th occurrence of the symbol $1$ in $Q_\gamma(1)Q_\gamma(2)\cdots$. Since
\[ \sum_{j=1}^{k}Q_\gamma(j) = \lfloor k\omega_\gamma\rfloor,\]
we have $k_m = \left\lceil\frac{m}{\omega_\gamma}\right\rceil$. Before the block $H(Q_\gamma(k_m))=10$, there are $k_m-1$ blocks, of which $m-1$ have length two. Hence the $m$-th occurrence of the symbol $1$ in the word on the left-hand side of \eqref{eq:H_Q_characteristic_rotation_word} is at position $r_m = k_m+m-1$. Since $\gamma = \omega_{\gamma} / (1 + \omega_{\gamma})$, we obtain
\[ r_m = \left\lceil\frac{m}{\gamma}\right\rceil-1.\]

On the other hand, we have
\[ \sum_{j=1}^{N}s_\gamma(j) = \lfloor(N+1)\gamma\rfloor,\]
so the $m$-th occurrence of the symbol $1$ in $s_\gamma(1)s_\gamma(2)\cdots$ is also at position $\left\lceil\frac{m}{\gamma}\right\rceil-1$. Thus the two binary words have exactly the same positions of their symbols $1$, proving \eqref{eq:H_Q_characteristic_rotation_word}.

Combining \eqref{eq:external_angle_Q_morphism}, \eqref{eq:H_Q_characteristic_rotation_word}, and \eqref{eq:characteristic_external_angle_rotation}, we obtain
\[ \vartheta_\gamma = 0.10s_\gamma(1)s_\gamma(2)\cdots.\]
Multiplication by $1/4$ shifts the binary expansion of $\theta_\gamma$ two places to the right, and addition of $1/2$ prefixes the digits $10$. Hence
\[ \vartheta_\gamma = \frac12+\frac{\theta_\gamma}{4} = T(\theta_\gamma),\]
as required.
\end{proof}

Bl\'e starts from the point of the main cardioid with internal argument $\gamma$ and external argument $\theta_\gamma$. He proves that the external parameter ray $R_M\bigl(T(\theta_\gamma)\bigr)$ lands at a real parameter $t_\gamma^{\mathrm B}$; see \cite[Theorem~1.3]{Ble_2204_external_arguments_and_invariant_measures}. Equivalently, $t_\gamma^{\mathrm B}$ is the common limit obtained from the even and odd continued-fraction convergents of $\gamma$; see \cite[Remark~3.9]{Ble_2204_external_arguments_and_invariant_measures}.

On the other hand, the standard kneading-to-external-angle correspondence for the real quadratic family associates to the minimum-type kneading sequence $\nu_\gamma^{\min}$ the characteristic external parameter angles $\vartheta_\gamma$ and  $1-\vartheta_\gamma$.

We can now identify the two quadratic parameters.

\begin{proposition}
\label{prop:Ble_quadratic_representative}
Let $\gamma \in (0,1/2) \setminus \Q$. Then $t_\gamma^{\mathrm B}=t_\gamma$, and the kneading map of $P_{t_\gamma^{\mathrm B}}$ is $Q_\gamma$.
\end{proposition}

\begin{proof}
The parameter ray $R_M(\vartheta_\gamma)$ lands at the quadratic representative $t_\gamma$.

By \Cref{lem:external_angle_Q_gamma}, $\vartheta_\gamma = T(\theta_\gamma)$. Hence the same external parameter ray lands at both $t_\gamma$ and $t_\gamma^{\mathrm B}$. Since an external ray has at
most one landing point, we have $t_\gamma=t_\gamma^{\mathrm B}$.

Finally, $P_{t_\gamma^{\mathrm B}}$ has the minimum-type kneading sequence $\nu_\gamma^{\min}$ and therefore kneading map $Q_\gamma$.
\end{proof}

Thus the class constructed by Bl\'e coincides, angle by angle, with the quadratic representatives of the long-branched Sturmian classes $\cS_\gamma$. Bl\'e's external-angle and cut-and-glue constructions provide a geometric description of these quadratic representatives, whereas the definition of $\cS_\gamma$ through the kneading map $Q_\gamma$ is purely combinatorial.


\section{Absolutely continuous invariant probability measures}
\label{sec:acim}

For the rest of this section, fix $\gamma\in(0,1/2)\setminus\Q$, $\ell\in(1,2]$, and a map $f\in\cS_\gamma^\ell$ with turning point $c$. By our standing assumptions, $f$ satisfies the contraction principle. Hence, by \Cref{cor:uniform_recurrence_long_branched_sturmian}, the turning point is recurrent and nonperiodic. Moreover, as proved in the preceding section, $f$ is nonrenormalizable. Consequently, the principal nest
\[ I_0\supset I_1\supset I_2\supset\cdots\ni c\]
is infinite and has infinitely many noncentral moments. By \Cref{thm:Shen_decay_of_geometry}, the map $f$ displays decay of geometry. The purpose of this section is to combine this geometric decay with arithmetic control of the central-cascade lengths.

Recall the noncentral moments $m(k)$ and the cascade lengths $L_k=m(k+1)-m(k)$ from \Cref{subsec:principal_nest_preliminaries}. Centrality of a principal-nest return is determined by the kneading data. Consequently, the cascade pattern is the same for every realization in $\cS_\gamma$.

The following estimate is proved by Bl\'e for the quadratic representative; see \cite[Lemma~5.2]{Ble_2204_external_arguments_and_invariant_measures}. Although stated there for a quadratic polynomial, the estimate depends only on the sequence of central and noncentral principal-nest returns, and therefore it extends to every realization in the same kneading class.

\begin{lemma}
\label{lem:continued_fraction_cascade_length}
Write $\gamma=[a_1,a_2,a_3,\ldots]$. Then, for every sufficiently large $k$,
\[ L_k \leq
    \begin{cases}
        a_{2k+1},&\text{if }a_1=2,\\[1mm]
        a_{2k-1},&\text{if }a_1>2.
    \end{cases}
\]
\end{lemma}

We next control the loss of geometry caused by one central pullback.

\begin{lemma}
\label{lem:central_pullback_general_order}
There exist constants $K_{\ell} \geq 1$ and $k_1(f)\geq1$ such that, for every $k\geq k_1(f)$ and every $m(k)\leq n\leq m(k+1)-2$, if $\mu_n\leq\frac12$, then
\begin{equation}
\label{eq:central_pullback_general_order}
        \mu_{n+1} \leq K_{\ell} \mu_n^{1/\ell}.
\end{equation}
\end{lemma}

\begin{proof}
For $k\geq1$, let $s_k$ be the first return time of $c$ to $I_{m(k)}$, and set $H_k = f^{s_k-1}$, and $G_k = f^{s_k} = H_k\circ f$. The returns at the levels $m(k),m(k)+1,\ldots,m(k+1)-2$ are central. If the first return time to $I_j$ is $s_k$ and this return is central, then $c_{s_k}\in I_{j+1}$. No return to $I_{j+1}$ can occur before time $s_k$, since such a return would also be an earlier return to $I_j$. Consequently, the first return time remains equal to $s_k$ throughout the entire cascade, including its terminal noncentral return. Thus $G_k|_{I_{j+1}}$ is the first return map from $I_{j+1}$ to $I_j$ for every $m(k)\leq j\leq m(k+1)-1$. 


By \Cref{prop:principal_nest_extension_distortion}, we have
\[ \sup_{x,y\in f(I_{m(k)+1})} \frac{|H_k'(x)|}{|H_k'(y)|} \leq D_{\ell}.\]
Since $I_{n+1}\subset I_{m(k)+1}$ for every $n\geq m(k)$, the same distortion bound applies to the restriction of $H_k$ to $f(I_{n+1})$.

Let $\phi$ and $\psi$ be the local critical normal-form coordinates from the definition of non-flatness. Thus $\phi(c) = 0 = \psi(f(c))$, and, after possibly reversing one of the coordinates, $F\= \psi\circ f\circ\phi^{-1}$ satisfies $|F(u)|=|u|^\ell$ for all $u$ sufficiently close to $0$.

The principal intervals shrink to $c$, and their images under $f$ shrink to $f(c)$. Since $\phi$, $\phi^{-1}$, $\psi$, and $\psi^{-1}$ are fixed $C^3$ diffeomorphisms for the map $f$, their derivative distortions on these intervals converge to $1$. Therefore, after increasing a map-dependent initial level, we may assume that
\begin{equation}
    \label{eq:critical_coordinates_small_distortion}
    \Dist(\phi|_{I_j}), \Dist(\phi^{-1}|_{\phi(I_j)}), \Dist(\psi|_{f(I_j)}),
    \Dist(\psi^{-1}|_{\psi(f(I_j))}) \leq \frac{4}{3}
\end{equation}
for every principal interval occurring below. The initial level at which \cref{eq:critical_coordinates_small_distortion} holds may depend on $f$, but the bound $4/3$ is fixed.

Put $\widehat I_j=\phi(I_j)$. For all sufficiently large $j$, the interval $I_j$ is contained in the domain of the critical normal-form coordinate $\phi$. Moreover, the central return branch on $I_j$ factors as a monotone map after the folding map $f$. The two boundary points of $I_j$ therefore have the same image under $f$. Since the normalized critical map $F=\psi\circ f\circ\phi^{-1}$ is even, it follows that $\widehat I_j:=\phi(I_j)=(-r_j,r_j)$ for some $r_j>0$. Define
\[ \widehat\mu_j = \frac{|\widehat I_{j+1}|}{|\widehat I_j|} = \frac{r_{j+1}}{r_j}.\]
The distortion bound for $\phi$ gives
\begin{equation}
    \label{eq:original_normalized_scaling_comparison}
    \frac34\mu_j \leq \widehat\mu_j \leq \frac43\mu_j
\end{equation}
for every sufficiently large $j$.

Define the normalized versions of $H_k$ and $G_k$ by $\widehat H_k = \phi\circ H_k\circ\psi^{-1}$ and $\widehat G_k = \phi\circ G_k\circ\phi^{-1}$. Since $G_k=H_k\circ f$, one has $\widehat G_k = \widehat H_k\circ F$. 
By \Cref{prop:principal_nest_extension_distortion} and \cref{eq:critical_coordinates_small_distortion},
\begin{equation}
    \label{eq:normalized_Hk_distortion}
    \Dist \left( \widehat H_k \bigm| F(\widehat I_{m(k)+1}) \right) \leq \left(\frac43\right)^2D_\ell
    = \frac{16}{9}D_\ell.
\end{equation}
Indeed, the two additional factors come from the derivative distortions
of $\phi$ and $\psi^{-1}$. Put $\widehat D_\ell = \frac{16}{9}D_\ell$.

Now let $m(k)\leq n\leq m(k+1)-2$ and suppose that $\mu_n\leq1/2$. By \cref{eq:original_normalized_scaling_comparison},
\[ \widehat\mu_n \leq \frac43\mu_n \leq \frac23.\]
The interval $\widehat G_k(\widehat I_{n+1})$ joins $\widehat G_k(0)\in\widehat I_{n+1}$ to an endpoint of $\widehat I_n$. Therefore
\[ |\widehat G_k(\widehat I_{n+1})| \ge r_n-r_{n+1} = r_n(1-\widehat\mu_n) \ge \frac{r_n}{3}.\]
Moreover, $\widehat G_k(\widehat I_{n+2}) \subset \widehat I_{n+1}$, and hence $|\widehat G_k(\widehat I_{n+2})| \le 2r_{n+1}.$ It follows that
\begin{equation}
    \label{eq:normalized_G_ratio_central_pullback}
    \frac{ |\widehat G_k(\widehat I_{n+2})|}{|\widehat G_k(\widehat I_{n+1})|} \le 6\widehat\mu_n.
\end{equation}

Using $\widehat G_k = \widehat H_k\circ F$ together with \cref{eq:normalized_Hk_distortion} and \cref{eq:normalized_G_ratio_central_pullback}, we obtain
\[ \frac{ |F(\widehat I_{n+2})|}{|F(\widehat I_{n+1})|} \le 6\widehat D_\ell\widehat\mu_n.\]
Since $|F(u)|=|u|^\ell$, one has $|F(\widehat I_j)|=r_j^\ell$. Consequently,
\[ \widehat\mu_{n+1}^\ell = \frac{r_{n+2}^\ell}{r_{n+1}^\ell} = \frac{|F(\widehat I_{n+2})|}{|F(\widehat I_{n+1})|}
    \le 6\widehat D_\ell\widehat\mu_n,\]
and therefore
\[ \widehat\mu_{n+1} \leq (6\widehat D_\ell)^{1/\ell} \widehat\mu_n^{1/\ell}.\]

Finally, applying \cref{eq:original_normalized_scaling_comparison} at levels $n$ and $n+1$ gives
\[  \mu_{n+1} \leq \frac{4}{3} \ \widehat{\mu}_{n+1} \le \frac{4}{3} (6\widehat{D}_\ell)^{1/\ell} \ \widehat\mu_n^{1/\ell} \leq \frac43 \left( 6\widehat D_\ell\frac43 \right)^{1/\ell} \mu_n^{1/\ell}.\]
Thus the conclusion holds with
\[ K_\ell = \frac43 \left( 8\widehat D_\ell \right)^{1/\ell} = \frac43 \left( \frac{128}{9}D_\ell \right)^{1/\ell}.\]
This constant depends only on $\ell$, since $D_\ell$ does. Enlarging the map-dependent initial index so that both \Cref{prop:principal_nest_extension_distortion} and \cref{eq:critical_coordinates_small_distortion} apply completes the proof.
\end{proof}

Put
\begin{equation}
\label{eq:epsilon_star_general_acip}
        \e_{\ell} \= \min\left\{\e(\ell),\frac12\right\}, \qquad \text{and} \qquad 
        M_{\ell} \= \max\left\{1,K_{\ell}^{\ell/(\ell-1)}\right\},
\end{equation}
where $\e(\ell)$ is the constant from \Cref{thm:Bruin_shen_vs_unimodal_acip_theorem}, and put
\begin{equation}
    \label{eq:k_2_level}
    k_2(f) \= \max \{k_0(f),k_1(f)\},
\end{equation}
where $k_0(f)$, and $k_1(f)$ are as in \cref{cor:uniform_decay_fixed_critical_order} and \cref{lem:central_pullback_general_order} respectively.

\begin{lemma}
\label{lem:propagation_noncentral_decay_general}
Let $k\geq k_2(f)$ and suppose that
\[ M_\ell \exp\left(-\frac{b_\ell k}{\ell^{L_k}}\right) \le \e_{\ell}. \]
Then, for every $0\leq r\leq L_k-1$,
\[ \mu_{m(k)+r} \le M_{\ell} \exp\left(-\frac{b_\ell k}{\ell^r}\right) \le \e_{\ell}.\]
\end{lemma}

\begin{proof}
We ague by induction on $r$. For $r=0$, the first inequality follows from \Cref{cor:uniform_decay_fixed_critical_order} and the fact that $M_{\ell} \geq 1$. Assume that the estimate holds for some $r$ with $0\leq r\leq L_k-2$. 
Since $0\leq r\leq L_k-1$, one has $\ell^r\leq\ell^{L_k}$, and hence
\[ M_\ell \exp\left(-\frac{b_\ell k}{\ell^r}\right) \le  M_\ell 
\exp\left(-\frac{b_\ell k}{\ell^{L_k}}\right) \le \e_\ell,\]
so \Cref{lem:central_pullback_general_order} applies. Hence
\[ \mu_{m(k)+r+1} \leq K_{\ell} \; \mu_{m(k)+r}^{1/\ell} \leq K_{\ell} \; M_{\ell} ^{1/\ell} \exp \left(-\frac{b_{\ell} k}{\ell^{r+1}}\right) \leq M_{\ell} \exp\left(-\frac{b_{\ell} k}{\ell^{r+1}}\right),\]
because $M_{\ell}^{1-1/\ell}\geq K_{\ell}$. This completes the induction.
\end{proof}

Now we are ready to prove \Cref{thm:acim_rotation_like}

\begin{proof}[Proof of \Cref{thm:acim_rotation_like}]
Define
\[ \eta_{\ell} \= \frac{b_{\ell}}{2 \log(M_{\ell} / \e_{\ell})}.\]
Since $M_{\ell} \ge 1$, and $\e_{\ell} \le 1/2$ this is a positive number. Suppose that
\[ \limsup_{n\to\infty} \frac{\ell^{a_{2n+1}}}{n} < \eta_{\ell},\]
and choose $\eta'$ strictly between this limsup and $\eta_{\ell}$. By \Cref{lem:continued_fraction_cascade_length}, for all sufficiently large $k$, we have $ \ell^{L_k} k^{-1} < \eta'$. When $a_1>2$, one has
\[ \frac{\ell^{L_k}}{k} \leq \frac{k-1}{k} \frac{\ell^{a_{2k-1}}}{k-1},\]
and the right-hand side has the same limiting upper bound as $\ell^{a_{2n+1}}/n$. Therefore
\[ \frac{k}{\ell^{L_k}} > \frac1{\eta_{\ell}} = \frac{2\log(M_{\ell}/\e_{\ell})}{b_{\ell}}.\]
Consequently,
\[ M_{\ell} \exp\left(-\frac{b_{\ell} k}{\ell^{L_k}}\right) < \e_{\ell}.\]
By \Cref{lem:propagation_noncentral_decay_general}, every sufficiently deep scaling factor satisfies $\mu_n\leq\e_{\ell} \leq \e(\ell)$. Then, \Cref{thm:Bruin_shen_vs_unimodal_acip_theorem} implies that $f$ admits an absolutely continuous invariant probability measure.
\end{proof}

\begin{remark}
For $\ell > 1$, if the odd partial quotients $a_{2n+1}$ are bounded, then
\[ \frac{\ell^{a_{2n+1}}}{n} \longrightarrow 0.\]
Thus, when $\ell=2$ and $f$ is the quadratic representative, \Cref{thm:acim_rotation_like} recovers
\cite[Theorem~1.4]{Ble_2204_external_arguments_and_invariant_measures}. More generally, the theorem also permits unbounded odd partial quotients. For example, its hypothesis is satisfied whenever
\[ a_{2n+1}=o(\log n).\]
\end{remark}


\section{Nonexistence of an acip}
\label{sec:non_existence_of_acip}

In this section, we prove \Cref{thm:general_arithmetic_nonexistence_acip}.

Fix $\gamma\in(0,1/2)\setminus\Q$ and $\ell>1$, and let $f\in\cS_\gamma^\ell$ have turning point $c$. Write $K(f)=e_1e_2e_3\cdots$ for the kneading sequence of the critical value $c_1$. Thus $e_j=\iota_0(c_j)=\iota_j(c)$, for every $j \ge 1$. For every $k\geq1$, define
\[ \cJ_k = \left\{ x\in I: \iota_j(x)=e_j \text{ for every }1\leq j\leq S_k-1 \right\},\]
and let $J_k$ be the closure of the connected component of $\cJ_k$ containing $c$.

For the cylinder sets below, we use the convention that $\iota_j(x)=C$ whenever $f^j(x)=c$. In particular, a point whose orbit meets $c$ at some time $1\leq j\leq S_k-1$ does not belong to $\cJ_k$.

\begin{lemma}
    \label{lem:critical_cylinder_connected}
For every $k\geq1$ and every $x\in\cJ_k$, one has $[c,x]\subset\cJ_k$, where $[c,x]$ denotes the closed interval with endpoints $c$ and $x$. Consequently, $\cJ_k$ is an interval containing $c$, and  $J_k=\overline{\cJ_k}$. Moreover, if $m\geq1$ satisfies $R(m)\geq S_k$, then $[c,f^m(c)]\subset J_k$. In particular,
\[ |J_k| \ge |f^m(c)-c|.\]
\end{lemma}

\begin{proof}
Fix $k\geq1$ and put $N=S_k-1$. Since the turning point is nonperiodic, its positive orbit does not meet $c$. Thus $c\in\cJ_k$.

Let $x\in\cJ_k$. The conclusion is immediate when $x=c$, so suppose that $x\neq c$ and let $K=[c,x]$. Notice that $K$ is contained in one of the two closed monotonicity intervals of $f$, with $c$ as one of its endpoints.

We first claim that $f^j(y)\neq c$ for every $y\in K$ and every $1\leq j\leq N$. Suppose otherwise. Choose the smallest integer $r\in\{1,\ldots,N\}$ for which there exists $y\in K$ satisfying $f^r(y)=c$. The point $y$ cannot be either endpoint of $K$. Indeed, the turning point $c$ is nonperiodic, while $x\in\cJ_k$ implies that $f^j(x)\neq c$ for every $1\leq j\leq N$. Hence $y\in\interior(K)$. 

By the minimality of $r$, $c\notin f^j(K)$ for every $1\leq j<r$. Since $f$ is strictly monotone on $K$ and on every interval $f^j(K)$, for $1\leq j<r$, it follows that $f^r|_K$ is strictly monotone. As $y\in\interior(K)$ and $f^r(y)=c$, strict monotonicity implies that $c$ lies strictly between $f^r(c)$ and $f^r(x)$.

On the other hand, since $x\in\cJ_k$, one has $\iota_r(x)= \iota_r(c) = e_r$. Thus $f^r(x)$ and $f^r(c)$ lie strictly on the same side of $c$, which contradicts the preceding conclusion. This proves the claim.

Now let $z\in K$. For every $1\leq j\leq N$, the set $f^j(K)$ is an interval that contains $f^j(c)$ and does not contain $c$. It is therefore contained entirely in one of the two components of $I\setminus\{c\}$. Consequently, $f^j(z)$ and $f^j(c)$ lie on the same side of $c$, and hence $\iota_j(z) = \iota_j(c) =e_j$ for every $1\leq j\leq N$. Thus $z\in\cJ_k$. Since $z\in K$ was arbitrary, we obtain $[c,x]\subset\cJ_k$.

It follows that $\cJ_k$ is an interval containing $c$, and thus every point of $\cJ_k$ is joined to $c$ by an interval contained in $\cJ_k$. Since $J_k$ was defined as the closure of the connected component of $\cJ_k$ containing $c$, we conclude that $J_k=\overline{\cJ_k}$.

Finally, suppose that $R(m)\geq S_k$. By the definition of the first disagreement time, $e_{m+j}=e_j$ for every $1\leq j\leq S_k-1$. Since $\iota_j\bigl(f^m(c)\bigr) = e_{m+j}$, we have $\iota_j\bigl(f^m(c)\bigr)=e_j$ for every $1\leq j\leq S_k-1$. Therefore $f^m(c)\in\cJ_k$. Applying the first part of the lemma with $x=f^m(c)$ gives
\[  [c,f^m(c)] \subset \cJ_k \subset J_k.\]
Taking lengths yields $|J_k| \ge |f^m(c)-c|$, as required.
\end{proof}

Put $E_k \= \interior(J_k\setminus J_{k+1})$, and $g_k \=|J_k|$. The sets $E_k$ may have two connected components, we denote by $|E_k|$ their total length. By definition $J_{k+1}\subset J_k$, and 
\[ \bigcap_{k\geq1}J_k=\{c\}.\]
Indeed, if the intersection is not a singleton, it contains a nondegenerate interval with $c$ as one of its boundary points. Then all iterates of $f$ would be monotone on that interval, so it would be a homterval. Such an interval is either wandering or eventually mapped into the basin of a periodic attractor. Both alternatives are excluded by our standing assumptions. Therefore $g_k \to 0$, and $|E_k|=g_k-g_{k+1}$.

We start recalling the finiteness criterion for the long-branched inducing scheme. The following theorem is due to Bruin \cite{Bruin1994_top_condition_for_the_existence_of_acip}. See also \cite[Proposition~2.9 and Corollary~2.10]{Ble_2204_external_arguments_and_invariant_measures}.

\begin{theorem}
\label{thm:long_branched_inducing_criterion}
Let $f$ be a long-branched $S$-unimodal map with recurrent nonperiodic turning point and no attracting periodic orbit. Let $J_k$ and $E_k$ be the kneading neighborhoods and layers defined above. Then the long-branched Markov inducing scheme admits an absolutely continuous invariant measure on its inducing domain. Its lift to the original system is finite if and only if
\begin{equation}
    \label{eq:long_branched_lift_integrability}
    \sum_{k\geq1}k|E_k|<+\infty.
\end{equation}
Consequently, $f$ admits an absolutely continuous invariant probability measure if and only if \cref{eq:long_branched_lift_integrability} holds.
\end{theorem}

We will use the following consequence of the previous theorem.

\begin{proposition}
\label{prop:real_kneading_neighborhood_acip_criterion}
The map $f$ admits an absolutely continuous invariant probability measure if and only if
\[ \sum_{k\geq1}|J_k|<+\infty.\]
\end{proposition}

\begin{proof}
By \Cref{thm:long_branched_inducing_criterion}, it suffices to prove that \eqref{eq:long_branched_lift_integrability} holds if and only if $\sum_{k\geq1}g_k<+\infty$.
Since $|E_k|=g_k-g_{k+1}$, we have
\begin{equation}
    \label{eq:summation_by_parts}
    \sum_{k=1}^{N}k|E_k| = \sum_{k=1}^{N}g_k-Ng_{N+1}.
\end{equation}
If $\sum_{k\geq1}g_k<+\infty$, the left-hand series converges. Conversely, if \eqref{eq:long_branched_lift_integrability} holds, then
\[ Ng_N \leq \sum_{k\geq N}k|E_k| \longrightarrow 0.\]
Letting $N\to\infty$ in \eqref{eq:summation_by_parts} give us
\[ \sum_{k\geq1}g_k = \sum_{k\geq1}k|E_k| <+\infty.\]
The conclusion follows.
\end{proof}

Let $\gamma = [a_1,a_2, \ldots]$ be the continued fraction expansion of $\gamma$, and let $\frac{p_n}{q_n} = [a_1, \ldots, a_n]$, be its $n$th convergent. The following proposition gives a dynamical criterion for nonexistence of an acip.
\begin{proposition}
\label{prop:dynamical_nonexistence_acip_rotation_like}
If
\[ \sum_{n\geq1} a_{2n+1}(q_{2n}-p_{2n}) |f^{q_{2n}}(c)-c| = +\infty,\]
then $f$ does not admit an absolutely continuous invariant probability measure.
\end{proposition}

\begin{proof}
For every $n\geq0$, put  $d_n=q_{2n+1}-p_{2n+1}$. By \Cref{thm:explicit_co_cutting_times_rotation_like}, we have $R(q_{2n}) = q_{2n+1} = S_{d_n}$. Hence \Cref{lem:critical_cylinder_connected} gives $[c,f^{q_{2n}}(c)] \subset J_{d_n}$. Therefore
\[ g_{d_n} = |J_{d_n}| \geq |f^{q_{2n}}(c)-c|.\]
Now, observe that 
\[ d_n-d_{n-1} = (q_{2n+1}-p_{2n+1}) -(q_{2n-1}-p_{2n-1}) = a_{2n+1}(q_{2n}-p_{2n}).\]
Since the intervals $J_k$ are nested, we have $g_k\geq g_{d_n}$ whenever $d_{n-1}<k\leq d_n$. Therefore
\[  \sum_{k\geq1}g_k \geq \sum_{n\geq1} (d_n-d_{n-1})g_{d_n} \geq 
\sum_{n\geq1} a_{2n+1}(q_{2n}-p_{2n})|f^{q_{2n}}(c)-c|.\]
The conclusion follows from \Cref{prop:real_kneading_neighborhood_acip_criterion}.
\end{proof}

The following quantitative trajectory-recurrence estimate is a consequence of \cite[Appendix]{Nowicki-Przytycki1998_top_inv_of_CE_property}.

\begin{lemma}
\label{lem:quantitative_trajectory_recurrence}
Let $f \colon I \to I$ be a $C^1$ unimodal map with a unique  critical point $c$. Assume that $c$ is nonperiodic, that $f$ has no attracting periodic orbit, and there are constants $\alpha \in (0,1]$ and $H \ge 1$ such that $|f'(x)| \le H |x-c|^\alpha$ for every $x \in I$. Then
\[ |f^m(c)-c| \geq e^{-1}\Lambda(\alpha,H)^{-m}\]
for every $m\geq1$, where
\[ \Lambda(\alpha,H) \= \exp\left( \frac{ \log2+\log H+\alpha-\log(1-e^{-1}) }{\alpha}\right).\]
\end{lemma}

As a consequence, we obtain an exponential lower bound for the recurrence of the turning point of every map under consideration.

\begin{lemma}
\label{lem:map_dependent_exponential_lower_recurrence}
There are constants $C_f>0$ and $\Lambda_f>1$ such that
\[ |f^m(c)-c| \geq C_f\Lambda_f^{-m}\]
for every $m\geq1$.
\end{lemma}

\begin{proof}
By nonflatness, there are local $C^3$ diffeomorphisms $\phi$ and $\psi$, with $\phi(c)=0 = \psi(f(c))$, such that $|\psi(f(x))|=|\phi(x)|^\ell$ near $c$. Differentiating the critical normal form away from $c$ shows that
\[ |f'(x)| \le H_f|x-c|^{\ell-1}\]
for all $x$ sufficiently close to $c$. Set $\alpha_f \= \min\{1,\ell-1\}.$ After an affine normalization of the phase interval and an increase of $H_f$ if necessary, one has
\[ |f'(x)| \le H_f|x-c|^{\alpha_f} \]
on the entire normalized interval.

Applying \Cref{lem:quantitative_trajectory_recurrence} gives
\[ |f^m(c)-c| \ge C_f\Lambda_f^{-m} \]
for suitable constants $C_f>0$ and $\Lambda_f>1$. Returning from the normalized interval only changes the multiplicative constant $C_f$.
\end{proof}

We now prove \Cref{thm:general_arithmetic_nonexistence_acip}

\begin{proof}[Proof of \Cref{thm:general_arithmetic_nonexistence_acip}]
Let $f\in\cS_\gamma^\ell$. By \Cref{lem:map_dependent_exponential_lower_recurrence},
\[ \sum_{n\geq1} a_{2n+1}(q_{2n}-p_{2n}) |f^{q_{2n}}(c)-c| \ge C_f \sum_{n\geq1} a_{2n+1}(q_{2n}-p_{2n}) \Lambda_f^{-q_{2n}}.\]
If the series on the right diverges, then \Cref{prop:dynamical_nonexistence_acip_rotation_like} applies. The factor $C_f>0$ does not affect divergence.
\end{proof}

Bl\'e's nonexistence criterion assumes that, for every $0<\lambda<1$,
\begin{equation}
    \label{eq:Ble_nonexistence_condition_comparison}
    \limsup_{n\to\infty} a_{2n+1}q_{2n}\lambda^{q_{2n}} =+\infty;
\end{equation}
see \cite[Theorem~1.5]{Ble_2204_external_arguments_and_invariant_measures}. Our condition uses only the single explicit scale $\Lambda_{f}^{-1}$ and asks for divergence of a series rather than terms with infinite limsup of its terms. Hence Bl\'e's hypothesis implies ours, but not conversely. We now move to the quadratic case $\ell = 2$, and construct some exaple reflecting this fact.

As explained in the preliminaries, we identify the quadratic representative with its maximum-type affine conjugate when applying results stated for maps in $\cS_\gamma$.

First, let us prove that in the quadratic case, we can choose the scale $\Lambda_f$ uniformly.

\begin{lemma}
    \label{lem:universal_quadratic_lower_recurrence}
Set
\[ \Lambda_{\rm rec} \= \frac{16e^2}{e-1}.\]
Let $P_t(x)=x^2+t$, with $t\in[-2,-1)$, and assume that the critical point $0$ is nonperiodic and that $P_t$ has no attracting periodic orbit. Then
\[|P_t^m(0)| \geq e^{-1}\Lambda_{\rm rec}^{-m}\]
for every $m\geq1$.
\end{lemma}

\begin{proof}
The invariant core of $P_t$ is $I_t=[t,t^2+t]$. Consider the orientation-preserving affine bijection $A_t\colon I_t \longrightarrow [0,1]$, defined by
\[  A_t(x)=\frac{x-t}{t^2},\]
and let $T_t = A_t\circ P_t\circ A_t^{-1}$. A direct calculation gives $T_t(u)=(1+tu)^2$. Its unique critical point is $z_t=A_t(0)=-\frac1t$, and
\[ T_t'(u) = 2t(1+tu) = 2t^2(u-z_t).\]
Since $t^2\leq4$, it follows that $|T_t'(u)| \leq 8|u-z_t|$ for every $u\in[0,1]$.

We may therefore apply \Cref{lem:quantitative_trajectory_recurrence} with $\alpha=1$ and $H=8$. The corresponding recurrence constant is
\[ \Lambda(1,8) = \exp\left( \log2+\log8+1-\log(1-e^{-1}) \right) = \frac{16e^2}{e-1} = \Lambda_{\rm rec}.\]
Consequently,
\[ |T_t^m(z_t)-z_t| \ge e^{-1}\Lambda_{\rm rec}^{-m}\]
for every $m\geq1$.

Finally, since $A_t\bigl(P_t^m(0)\bigr)=T_t^m(z_t)$, $A_t(0)=z_t$, and $(A_t^{-1})' = t^2$, we obtain
\[ |P_t^m(0)| = t^2|T_t^m(z_t)-z_t|.\]
Since $t^2\geq1$, the desired estimate follows.
\end{proof}

Let $\gamma\in(0,1/2)\setminus\Q$, with $\gamma=[a_1,a_2,a_3,\ldots]$, and convergents $\frac{p_n}{q_n}$, and let $P_{t_\gamma} \in \cS_{\gamma}^2$. 

\begin{corollary}
\label{cor:uniform_quadratic_nonexistence_acip}
If
\begin{equation}
    \label{eq:uniform_quadratic_nonexistence_series}
    \sum_{n\geq1} a_{2n+1}(q_{2n}-p_{2n}) \Lambda_{\rm rec}^{-q_{2n}} = +\infty,
\end{equation}
then $P_{t_\gamma}$ does not admit an absolutely continuous invariant probability measure.
\end{corollary}

\begin{proof}
By \Cref{lem:universal_quadratic_lower_recurrence}, we have
\[ \sum_{n\geq1} a_{2n+1}(q_{2n}-p_{2n}) |P_{t_\gamma}^{q_{2n}}(0)| \geq e^{-1} \sum_{n\geq1}
    a_{2n+1}(q_{2n}-p_{2n}) \Lambda_{\rm rec}^{-q_{2n}} = +\infty.\]
The conclusion follows from \Cref{prop:dynamical_nonexistence_acip_rotation_like}.
\end{proof}

Hypothesis \eqref{eq:Ble_nonexistence_condition_comparison} implies \Cref{eq:uniform_quadratic_nonexistence_series}. Indeed, take $\lambda=\Lambda_{\rm rec}^{-1}$. By \eqref{eq:best_approx_convergents} one has  $ q_{2n}-p_{2n} > \frac{q_{2n}}{2}$.
Therefore
\[  a_{2n+1}(q_{2n}-p_{2n}) \Lambda_{\rm rec}^{-q_{2n}} > 
    \frac12 a_{2n+1}q_{2n} \Lambda_{\rm rec}^{-q_{2n}}.\]
Under Bl\'e's hypothesis \eqref{eq:Ble_nonexistence_condition_comparison}, the terms on the right have infinite limsup, so the series in \Cref{eq:uniform_quadratic_nonexistence_series} diverges. The converse is false, as shown by the following example.

\begin{example}
\label{ex:non_existence_beyond_ble}
Set $a_1=2$ and $a_{2n}=1$ for every $n\geq1$. Define the odd coefficients recursively by
\begin{equation}
    \label{eq:noncircular_large_odd_coefficients}
    a_{2n+1} = \left\lceil \frac{\Lambda_{\rm rec}^{q_{2n}}}{q_{2n}} \right\rceil,
\end{equation}
for every $n \ge 1$. This recursion is well defined: when $a_{2n+1}$ is chosen, the denominator $q_{2n}$ depends only on the previously chosen coefficients $a_1,\ldots,a_{2n}$. Moreover, the constant $\Lambda_{\rm rec}$ is universal and does not depend on the angle being constructed.

Let $\gamma=[a_1,a_2,a_3,\ldots]$. Then $\gamma\in(0,1/2)\setminus\Q$. Since $p_{2n}/q_{2n}<\gamma$, we have
\[ \frac{q_{2n}-p_{2n}}{q_{2n}} = 1-\frac{p_{2n}}{q_{2n}} > 1-\gamma > \frac12.\]
Using \Cref{eq:noncircular_large_odd_coefficients}, it follows that
\[ a_{2n+1}(q_{2n}-p_{2n}) \Lambda_{\rm rec}^{-q_{2n}} \geq \frac{\Lambda_{\rm rec}^{q_{2n}}}{q_{2n}} (q_{2n}-p_{2n})
    \Lambda_{\rm rec}^{-q_{2n}} = \frac{q_{2n}-p_{2n}}{q_{2n}} > \frac12.\]
Consequently,
\[ \sum_{n\geq1} a_{2n+1}(q_{2n}-p_{2n}) \Lambda_{\rm{rec}}^{-q_{2n}} = +\infty.\]
By \Cref{cor:uniform_quadratic_nonexistence_acip}, the quadratic representative $P_{t_\gamma}$ does not admit an absolutely continuous invariant probability measure.

On the other hand, this angle does not satisfy Bl\'e's condition. Indeed, put $\lambda_0 = \frac{1}{2}\Lambda_{\rm{rec}}^{-1}  \in (0,1)$. Since $\lceil x\rceil\leq x+1$, we obtain
\[  a_{2n+1}q_{2n}\lambda_0^{q_{2n}} \leq \left( \frac{\Lambda_{\rm rec}^{q_{2n}}}{q_{2n}} +1 \right)
    q_{2n}\lambda_0^{q_{2n}} = \left( \Lambda_{\rm rec}\lambda_0 \right)^{q_{2n}} + q_{2n}\lambda_0^{q_{2n}}
    = 2^{-q_{2n}} + q_{2n}(2\Lambda_{\rm rec})^{-q_{2n}} \longrightarrow0. \]
Thus
\[ \limsup_{n\to\infty} a_{2n+1}q_{2n}\lambda_0^{q_{2n}} = 0,\]
whereas Bl\'e's hypothesis requires this limsup to be $+\infty$ for every $0 < \lambda < 1$. Hence the series criterion is strictly weaker than Bl\'e's condition.
\end{example}


\section{Failure of the Collet--Eckmann condition}
\label{sec:non_Collet_Eckmann}

In this section we prove that no finite-order $S$-unimodal realizationof an irrational long-branched Sturmian kneading class satisfies the Collet--Eckmann condition.

The proof has two parts. First, we show that the Collet--Eckmann condition forces the first-disagreement times \eqref{eq:first_symbolic_disagreement_time} to have bounded Ces\`aro averages. The reason is that agreement between the kneading sequence and its shift by $m$ produces a monotonicity interval joining $c$ to $f^m(c)$. Exponential shrinking of pullbacks then makes this interval exponentially small in $R(m)$. Nonflatness converts this proximity to the turning point into an upper bound on the derivative at $f^m(c)$.

Second, we prove that the Ces\`aro averages of $R(m)$ diverge for every long-branched Sturmian kneading sequence. This is a purely combinatorial consequence of the mechanical coding of an irrational rotation. Long prefixes of the Sturmian word occur after the symbol $1$ with frequencies bounded below by certain one-sided closest-return gaps of the rotation, and the sum of these gaps diverges.

\subsection{Collet--Eckmann condition and mean agreement times}
\label{subsec:CE_mean_agreement_times}

Let $f$ be an $S$-unimodal map with nonperiodic turning point $c$. Let $K(f)=e_1e_2e_3\cdots$ be the kneading sequence of the critical value. Recall that $R(m) = \min\{r\geq1:e_{m+r}\neq e_r\}$.

For $m\geq1$, let $V_m$ denote the closed interval with endpoints $c$ and $c_m$.

\begin{lemma}
\label{lem:agreement_interval_is_monotone}
For every $m\geq1$, the iterate $f^{R(m)}\big|_{V_m}$ is monotone.
\end{lemma}

\begin{proof}
Suppose otherwise. Then there exist an integer $1\leq r<R(m)$ and a point $y\in\operatorname{int}(V_m)$ such that $f^r(y)=c.$ Choose $r$ minimal with this property.

The interval $V_m$ is contained in one of the two closed monotonicity intervals of $f$, with $c$ as an endpoint. By the minimality of $r$, none of the intervals $f^j(V_m)$, with $1\leq j<r$, contains $c$. It follows that $f^r\big|_{V_m}$ is monotone. Since $y$ belongs to the interior of $V_m$ and $f^r(y)=c$, the point $c$ lies strictly between the two endpoint images $f^r(c)=c_r$  and $f^r(c_m)=c_{m+r}$. Consequently, $c_r$ and $c_{m+r}$ lie on different sides of $c$.

On the other hand, $r<R(m)$ implies $e_{m+r}=e_r$. Since the turning point is nonperiodic, neither $c_r$ nor $c_{m+r}$ is equal to $c$. Thus the equality of their itinerary symbols means that they lie on the same side of $c$, a contradiction. Therefore $f^{R(m)}|_{V_m}$ is monotone.
\end{proof}

We will use the following lemma \cite{Nowicki1988_positive_lyap_exp_for_crit_values}

\begin{lemma}[Exponential shrinking of monotonicity intervals]
\label{lem:CE_lap_shrinking}
Let $f$ be an $S$-unimodal map satisfying the Collet--Eckmann condition. Then there exist constants $C_0>0$ and $\xi\in(0,1)$ such that, for every $n\geq1$ and every interval $W\subset I$ on which $f^n$ is monotone, one has $|W|\leq C_0\xi^n.$
\end{lemma}

We also use the following logarithmic recurrence estimate. It is an immediate consequence of the trajectory-recurrence estimate of Nowicki and Przytycki \cite[Lemma~1.1, equation~(1.1)]{Nowicki-Przytycki1998_top_inv_of_CE_property}.

\begin{lemma}
\label{lem:NP_logarithmic_recurrence_bound}
Let $f$ be an $S$-unimodal map satisfying the Collet--Eckmann condition, with turning point $c$. After an affine normalization of the phase interval, there exists a constant $C_f>0$ such that
\[ \sum_{m=1}^{N} \log\frac{1}{|c_m-c|}\leq C_f N\]
for every $N\geq1$.
\end{lemma}

We now combine the two lemmas above to obtain a uniform bound on the Ces\'aro averages of the first-disagreement times.

\begin{proposition}
\label{prop:CE_implies_bounded_mean_agreement}
Let $f$ be an $S$-unimodal map whose turning point is nonflat of finite order $\ell>1$. If $f$ satisfies the Collet--Eckmann condition, then
\[ \sup_{N\geq1} \frac{1}{N} \sum_{m=1}^{N}R(m) <+\infty.\]
\end{proposition}

\begin{proof}
After an affine change of coordinates, we may assume that the phase interval has length one. This does not change the function $R$ or the Collet--Eckmann property.

Let $C_0\geq1$ and $\xi\in(0,1)$ be the constants from \Cref{lem:CE_lap_shrinking}. By
\Cref{lem:agreement_interval_is_monotone}, $f^{R(m)}$ is monotone on the interval $V_m=[c,c_m]$. Hence $|c_m-c| = |V_m| \leq C_0\xi^{R(m)}$. Therefore
\[ R(m) \le \frac{\log C_0+\log |c_m-c|^{-1}}{\log(\xi^{-1})}.\]
Summing over $1\leq m\leq N$ and using \Cref{lem:NP_logarithmic_recurrence_bound}, we obtain
\[ \sum_{m=1}^{N}R(m) \le \frac{N\log C_0+C_fN}{\log(\xi^{-1})}.\]
Consequently,
\[ \frac1N\sum_{m=1}^{N}R(m) \le \frac{\log C_0+C_f}{\log(\xi^{-1})}\]
for every $N\geq1$, which proves the proposition.
\end{proof}


\subsection{Mean agreement times for Sturmian kneading sequences}
\label{subsec:Sturmian_mean_agreement_times}

Fix an irrational number $\gamma\in(0,1/2)$ and put $\omega_\gamma = \gamma/ (1-\gamma) \in (0,1)$. For $n \ge 1$, let $w_n=Q_{\gamma}(1)Q_{\gamma}(2) \cdots Q_{\gamma}(n)$ be the prefix of length $n$ of the lower mechanical word $Q_{\gamma}(1) Q_{\gamma}(2) Q_{\gamma}(3)\cdots$.

We first record an elementary property of the one-sided closest returns of the irrational rotation of angle $\omega$. Put $\{x\} := x-\lfloor x\rfloor$ for the fractional part of $x$, and define
\[ u_n(\omega) := \min_{1\leq r\leq n} \bigl(1-\{r\omega\}\bigr).\]
\begin{lemma}
\label{lem:one_sided_rotation_gaps_not_summable}
Let $\omega \in (0,1) \setminus \Q$. We have
\[ \sum_{n\geq1}u_n (\omega) = +\infty.\]
\end{lemma}

\begin{proof}
Fix $\omega \in (0,1) \setminus \Q$, and let $u_n = u_n(\omega)$. Since the orbit of an irrational rotation is dense, the sequence $(u_n)_{n \ge 1}$ is positive, nonincreasing, and converges to zero.

Let $1\leq r_1<r_2<r_3<\cdots$ be the successive record times at which $u_n$ decreases strictly, and put $\bfu_k := 1-\{r_k\omega\}$, and $\bfw_k := \lceil r_k\omega\rceil$. Thus $\bfu_k = \bfw_k - r_k\omega$, and $u_n = \bfu_k$ whenever  $r_k\leq n<r_{k+1}$.

Since $\bfu_{k+1} < \bfu_k$ and $r_k < r_{k+1}$, we have
\[ \frac{\bfw_k}{r_k} = \omega+\frac{\bfu_k}{r_k} > \omega+\frac{\bfu_{k+1}}{r_{k+1}} = \frac{\bfw_{k+1}}{r_{k+1}}.\]
It follows that $\bfw_k r_{k+1} - \bfw_{k+1} r_k$ is a positive integer. Hence
\[  1 \le \bfw_k r_{k+1} - \bfw_{k+1} r_k = r_{k+1} \bfu_k - r_k \bfu_{k+1}.\]
In particular, $r_{k+1} \bfu_k\geq1$. Since $u_{r_{k+1}-1} = \bfu_k$, we obtain
\[ (r_{k+1}-1) u_{r_{k+1}-1} = (r_{k+1}-1) \bfu_k \ge 1-\bfu_k.\]
As $k\to\infty$, the right-hand side tends to $1$.

On the other hand, if the positive series $\sum_nu_n$ converged, then $nu_n\to0$. Indeed,
\[ \frac n2\,u_n \le \sum_{j=\lceil n/2\rceil}^{n}u_j.\]
The right hand side above goes to $0$ as $n \to \infty$. This contradicts the preceding lower bound along the subsequence $n=r_{k+1}-1$. Therefore
\[ \sum_{n\geq1}u_n=+\infty.\]
\end{proof}

In the following lemma we prove that the numbers $u_n(\omega)$ give lower bounds for the frequencies of certain factors in the mechanical word.

\begin{lemma}
\label{lem:frequency_of_one_followed_by_Sturmian_prefix}
For all sufficiently large $n$,
\[ \liminf_{k\to\infty} \frac{1}{k} \#\left\{ 1\leq j\leq k: Q_{\gamma}(j) Q_{\gamma}(j+1) \cdots Q_{\gamma}(j+n) = 1 Q_{\gamma}(1) Q_{\gamma}(2) \cdots Q_{\gamma}(n) \right\} \ge u_n(\omega_{\gamma}).\]
\end{lemma}

\begin{proof}
Let $\gamma \in (0,1/2) \setminus \Q$, nd write $u_n = u_n(\omega_{\gamma})$. For $x\in[0,1)$ and $i\geq1$, define the intercept-$x$ mechanical coding
\[ \cQ_i(x) := \lfloor x+i \omega_{\gamma} \rfloor - \lfloor x+(i-1) \omega_{\gamma} \rfloor.
\]
Thus $\cQ_i(0)= Q_{\gamma}(i)$ for every $i \ge 1$.

Fix $n\geq1$. If $0\leq x<u_n$, then $x < u_n \le 1-\{r\omega_{\gamma} \}$ for every $1\leq r\leq n$. Consequently, $\lfloor x+r \omega_{\gamma} \rfloor = \lfloor r \omega_{\gamma} \rfloor$. It follows that
\begin{equation}
    \label{eq:mechanical_prefix_constant_near_zero}
    \cQ_1(x) \cQ_2(x)\cdots \cQ_n(x) = Q_{\gamma}(1) Q_{\gamma}(2) \cdots Q_{\gamma}(n) = w_n
\end{equation}
for every $x\in[0,u_n)$.

Since $u_n\to0$, as $n \to \infty$, we can choose $n$ large enough so that $u_n < \omega_{\gamma}$. For $0\leq x<u_n$, put $y = 1- \omega_{\gamma} +x$. Then $y\in[0,1)$ and $\cQ_1(y)=1$. Moreover, for every $1\leq i\leq n$,
\[ \cQ_{i+1}(y) = \lfloor y+(i+1) \omega_{\gamma} \rfloor - \lfloor y+i \omega_{\gamma} \rfloor
    = \lfloor 1+x+i \omega_{\gamma} \rfloor - \lfloor 1+x+(i-1) \omega_{\gamma} \rfloor = \cQ_i(x).\]
Together with \cref{eq:mechanical_prefix_constant_near_zero}, this shows that
\[ \cQ_1(y) \cQ_2(y) \cdots \cQ_{n+1}(y) =1 Q_{\gamma}(1) Q_{\gamma}(2) \cdots Q_{\gamma}(n) \]
for every $y \in [1-\omega_{\gamma},1-\omega_{\gamma} + u_n)$.

For $j\geq1$, set $x_j=\{(j-1)\omega_{\gamma}\}$. A direct calculation gives $Q_{\gamma}(j+i) =        \cQ_{i+1}(x_j)$, for every $0\leq i\leq n$. Therefore if $x_j\in [ 1- \omega_{\gamma}, 1 -\omega_{\gamma} +u_n)$ then
\[ Q_{\gamma}(j) Q_{\gamma}(j+1) \cdots Q_{\gamma}(j+n) = 1 Q_{\gamma}(1) Q_{\gamma}(2) \cdots Q_{\gamma}(n).\]

By equidistribution of the orbit of an irrational rotation, equivalently by unique ergodicity of the irrational rotation,
\[ \lim_{k\to\infty} \frac{1}{k} \#\{1\leq j\leq k: x_j \in [1 -\omega_{\gamma}, 1 -\omega_{\gamma} +u_n)\} = |[1 -\omega_{\gamma}, 1 -\omega_{\gamma} +u_n)| = u_n. \]
Since every such visit produces the required factor, the conclusion follows.
\end{proof}

We now transfer these occurrences in the mechanical word to long self-agreements of the unimodal kneading sequence.

\begin{proposition}
\label{prop:Sturmian_mean_agreement_diverges}
Let $f\in\cS_\gamma$, and let $R(m)$ be the first-disagreement function of its kneading sequence. Then
\[ \lim_{M\to\infty} \frac{1}{M} \sum_{m=1}^{M}R(m) = +\infty.\]
\end{proposition}

\begin{proof}
Since we are using the convention that the turning point is a maximum, by \Cref{prop:cutting_times_kneading_sequence_sturmian}, the kneading sequence has the form $K(f) = 1\, \Phi(Q_{\gamma}(1))\Phi(Q_{\gamma}(2))\Phi(Q_{\gamma}(3))\cdots$, where $\Phi(0)=0$, and $\Phi(1)=11$. Moreover,
\[ S_j = 1+\sum_{i=1}^{j}|\Phi(Q_{\gamma}(i))| = 1+j+\lfloor j\, \omega_{\gamma} \rfloor.\]

For $j\geq1$, put $m_j=S_j-1$. Since $0<\omega<1$,
\begin{equation}
    \label{eq:mj_less_than_twice_j}
    m_j = j+\lfloor j\, \omega_{\gamma} \rfloor< 2j.
\end{equation}
In particular, the integers $m_j$ are distinct and $m_j\leq2N$ whenever $j\leq N$.

Suppose that, for some $j,n\geq1$, we have
\begin{equation}
    \label{eq:occurrence_one_prefix_assumption}
    Q_{\gamma}(j) Q_{\gamma}(j+1) \cdots Q_{\gamma}(j+n) = 1 Q_{\gamma}(1) Q_{\gamma}(2)\cdots Q_{\gamma}(n).
\end{equation}
In particular, $Q_{\gamma}(j)=1$, so $\Phi(Q_{\gamma}(j))=11$. The symbol at position $S_j$ of $K(f)$ is therefore the final symbol $1$ of the block $\Phi(Q_{\gamma}(j))$. By \cref{eq:occurrence_one_prefix_assumption}, the sequence beginning at this symbol is
\[ 1\, \Phi(Q_{\gamma}(j+1)) \Phi(Q_{\gamma}(j+2)) \cdots\Phi(Q_{\gamma}(j+n))
    = 1\, \Phi(Q_{\gamma}(1))\Phi(Q_{\gamma}(2))\cdots\Phi(Q_{\gamma}(n)).\]
The word on the right is the prefix of $K(f)$ of length $S_n$. Consequently,
\begin{equation}
    \label{eq:factor_occurrence_gives_long_R}
    R(m_j)>S_n.
\end{equation}

Choose $n_0$ so that \Cref{lem:frequency_of_one_followed_by_Sturmian_prefix} applies for every $n\geq n_0$. For $j,n\geq1$, write
\[ \chi_{j,n} =
    \begin{cases}
        1, \text{ if } Q_{\gamma}(j) Q_{\gamma}(j+1) \cdots Q_{\gamma}(j+n)
        =1 Q_{\gamma}(1) Q_{\gamma}(2) \cdots Q_{\gamma}(n),\\
        0,\text{ otherwise }.
    \end{cases}
\]

Fix $L\geq n_0$. For each fixed $j$, if $\chi_{j,n}=1$, then $\chi_{j,r}=1$ for every $r\leq n$. Suppose that $n_j = \max\{n_0\leq n\leq L:\chi_{j,n}=1\}$ exists, then \cref{eq:factor_occurrence_gives_long_R} gives $R(m_j)>S_{n_j}\geq n_j+1$. It follows that
\[ \sum_{n=n_0}^{L}\chi_{j,n} \le R(m_j).\]
The same inequality is trivial if the defining set for $n_j$ is empty.

Using \cref{eq:mj_less_than_twice_j} and the fact that the $m_j$ are distinct, we obtain
\[ \sum_{m=1}^{2N}R(m) \ge \sum_{j=1}^{N}R(m_j) \ge \sum_{j=1}^{N} \sum_{n=n_0}^{L}\chi_{j,n}.\]
Divide by $2N$ and let $N\to\infty$. Since the sum over $n$ is finite, \Cref{lem:frequency_of_one_followed_by_Sturmian_prefix} gives
\[  \liminf_{N\to\infty} \frac{1}{2N}\sum_{m=1}^{2N}R(m) \ge \frac{1}{2} \sum_{n=n_0}^{L}
    \liminf_{N\to\infty}\frac{1}{N}\sum_{j=1}^{N}\chi_{j,n} \ge \frac{1}{2} \sum_{n=n_0}^{L}u_n(\omega_{\gamma}).
\]
By \Cref{lem:one_sided_rotation_gaps_not_summable}, the expression on the right tends to $+\infty$ as $L\to\infty$. Hence
\[ \liminf_{N\to\infty} \frac{1}{2N} \sum_{m=1}^{2N}R(m) = +\infty.\]

Finally, for an arbitrary $M \in \N$, put $N=\lfloor M/2\rfloor$. Since $R(m)\geq1$,
\[ \frac{1}{M}\sum_{m=1}^{M}R(m) \ge \frac{2N}{2N+1} \left( \frac{1}{2N}\sum_{m=1}^{2N}R(m) \right).\]
This proves the proposition.

In the minimum-type convention first-disagreement time is the same, so the same conclusion holds in that convention.
\end{proof}

We can now prove \Cref{thm:no_CE_long_branched_sturmian}.

\begin{proof}[Proof of \Cref{thm:no_CE_long_branched_sturmian}]
Suppose that $f\in\cS_\gamma^\ell$ satisfies the Collet--Eckmann condition. By
\Cref{prop:CE_implies_bounded_mean_agreement},
\[ \sup_{N\geq1} \frac{1}{N} \sum_{m=1}^{N}R(m) < +\infty.\]
On the other hand, the kneading sequence of $f$ is the long-branched Sturmian kneading sequence of angle $\gamma$. Hence \Cref{prop:Sturmian_mean_agreement_diverges} gives
\[ \frac{1}{N} \sum_{m=1}^{N}R(m) \longrightarrow +\infty.\]
This is a contradiction. Therefore $f$ does not satisfy the Collet--Eckmann condition.
\end{proof}


\subsection{Persistent recurrence and Bl\'e's argument}
\label{subsec:Ble_persistent_recurrence_gap}

We briefly point out a gap in the persistent-recurrence argument used in \cite[Lemma~4.1]{Ble_2204_external_arguments_and_invariant_measures}. For a quadratic polynomial, Yoccoz's $\tau$-function may be described from the critical tableau by
\[ \tau(j) = \max\{m<j:c_{j-m}\in Y^{(m)}(c)\},\]
where $Y^{(m)}(c)$ is the critical puzzle piece of depth $m$.

In the proof of Lemma~4.1, Bl\'e establishes that $c_{q_{2n}+i q_{2n+1}}\notin Y^{(m)}(c)$ for $m\ge q_{2n+1}+1$ and $1\le i\le a_{2n+2}-1$, and then concludes that $q_{2n-1}\le\tau(j)$, for every $q_{2n+1} + 1 \le j \le q_{2n+3}$. The latter conclusion does not follow from the former. Indeed, a lower bound $\tau(j)\ge q_{2n-1}$ requires the existence, for each such $j$, of a critical tableau position $(m,j-m)$ with $m\ge q_{2n-1}$, whereas the preceding argument only establishes that certain specified tableau positions are noncritical. No such critical position is produced.

The missing step is particularly apparent when $a_{2n+2}=1$, for every $n \ge 1$. In that case the range $1\le i\le a_{2n+2}-1$ is empty, so the preceding exclusion is vacuous, while the same nontrivial lower bound for $\tau(j)$ is still asserted. Moreover, the construction preceding the conclusion involves only the semiconvergent block $q_{2n}+kq_{2n+1}$, which terminates at $q_{2n+2}$, whereas the asserted estimate extends to $q_{2n+3}$ without treating the intervening continued-fraction block.

Consequently, the argument in Lemma~4.1 does not establish persistent recurrence, and hence the proof of the non-Collet--Eckmann assertion in \cite[Theorem~1.3(ii)]{Ble_2204_external_arguments_and_invariant_measures} is incomplete.

This observation is also consistent with \cite[Theorem~6.2]{alvin-Cinic2026-Recurrence_symbolic_dynamics_and_wild_attractors}, which shows that a recurrent nonperiodic turning point with bounded kneading map cannot be persistently recurrent. Since Bl\'e proves in \cite[Proposition~5.5]{Ble_2204_external_arguments_and_invariant_measures} that the kneading map of these quadratic representatives satisfies $Q(k)\leq1$, they belong to this long-branched regime.


\bibliographystyle{alpha}
\bibliography{Biblio}
\end{document}